\documentclass[11pt]{amsart}

\usepackage{amssymb,bbold,mathtools}
\usepackage{color}
\usepackage{mathrsfs}
\usepackage{mathtools}
\usepackage{hyperref}
\usepackage{enumitem}
\usepackage{cleveref}
\usepackage[charter]{mathdesign}

\allowdisplaybreaks

\theoremstyle{plain}
\newtheorem{theorem}{Theorem}[section]
\newtheorem{proposition}[theorem]{Proposition}
\newtheorem{corollary}[theorem]{Corollary}
\newtheorem{lemma}[theorem]{Lemma}

\theoremstyle{definition}
\newtheorem{remark}[theorem]{Remark}
\newtheorem{example}[theorem]{Example}
\newtheorem{definition}[theorem]{Definition}

\newcommand{\abs}[1]{\lvert#1\rvert}
\newcommand{\norm}[1]{\lVert#1\rVert}
\newcommand{\uppars}[1]{\textup{(}#1\textup{)}}

\newcommand{\pos}[1]{{#1}^+}

\newcommand{\oc}{\mathrm{oc}}
\newcommand{\odual}[1]{{#1}^{\thicksim}}
\newcommand{\ocdual}[1]{{#1}^\thicksim_{\oc}}

\newcommand{\Ell}{\mathrm{L}}
\newcommand{\dc}{\mathrm{d}}

\renewcommand{\o}{\mathrm{o}}
\newcommand{\uo}{\mathrm{uo}}
\newcommand{\unb}{\mathrm{u}}

\newcommand{\conv}{\xrightarrow}
\newcommand{\convwithoverset}[1]{\conv{#1}}
\newcommand{\oconv}{\convwithoverset{\o}}
\newcommand{\uoconv}{\convwithoverset{\uo}}
\newcommand{\tauconv}{\convwithoverset{\tau}}

\newcommand{\sumet}{solidly submetrisable}
\newcommand{\locsumet}{locally solidly submetrisable}
\newcommand{\locsucommet}{locally interval complete solidly submetrisable}

\newcommand{\uoadhtext}{uo-adherence}

\newcommand{\suoadhtext}{$\sigma$-uo-adherence}

\newcommand{\adh}[2]{a_{#1}(#2)}
\newcommand{\sadh}[2]{a_{\sigma\text{-}#1}(#2)}

\newcommand{\uoadh}[1]{\adh{\uo}{#1}}

\newcommand{\suoadh}[1]{\sadh{\uo}{#1}}

\newcommand{\uoclos}[1]{\overline{#1}^{\uo}}
\newcommand{\suoclos}[1]{\overline{#1}^{\sigma\textup{-}\uo}}

\setlist[enumerate,1]{label=\textup{(\arabic*)},ref=\textup{(\arabic*)}}
\setlist[enumerate,2]{label=\textup{(\alph*)},ref=\textup{(\alph*)}}
\setlist[enumerate,3]{label=\textup{(\roman*)},ref=\textup{(\roman*)}}
\setlist[enumerate,4]{label=\textup{(\Alph*)},ref=\textup{(\Alph*)}}

\crefname{theorem}{Theorem}{Theorems}
\crefname{proposition}{Proposition}{Propositions}
\crefname{lemma}{Lemma}{Lemmas}
\crefname{corollary}{Corollary}{Corollaries}
\crefname{definition}{Definition}{Definitions}
\crefname{example}{Example}{Examples}
\crefname{remark}{Remark}{Remarks}
\crefname{section}{Section}{Sections}
\crefname{subsection}{Section}{Sections}
\crefname{subsubsection}{Section}{Sections}
\crefname{equation}{equation}{equations}

\begin{document}
\title[Embedded unbounded order convergent sequences]{Embedded unbounded order convergent sequences in topologically convergent nets in vector lattices}

\author{Yang Deng}
\address{Yang Deng; School of Economic Mathematics, Southwestern University of Finance and Economics, Chengdu, Sichuan, 611130 PR China}
\email{dengyang@swufe.edu.cn}

\author{Marcel de Jeu}
\address{Marcel de Jeu; Mathematical Institute, Leiden University, P.O.\ Box 9512, 2300 RA Leiden, the Netherlands;
	and Department of Mathematics and Applied Mathematics, University of Pretoria, Cor\-ner of Lynnwood Road and Roper Street, Hatfield 0083, Pretoria, South Africa}
\email{mdejeu@math.leidenuniv.nl}



\begin{abstract}
We show that, for a class of locally solid topologies on vector lattices, a topologically convergent net has an embedded sequence that is unbounded order convergent to the same limit. Our  result implies, and often improves, many of the known results in this vein in the literature. A study of metrisability and submetrisability of locally solid topologies on vector lattices is included.
\end{abstract}

\maketitle

\section{Introduction and overview}\label{sec:introduction_and_overview}

\noindent
Let $(X,\Sigma,\mu)$ be a measure space, and let $\Ell^0(X,\Sigma,\mu)$ denote the vector lattice of  measurable functions on~$X$, with identification of two functions that are $\mu$-almost everywhere equal. It is a classical result that a sequence in~$\Ell^0(X,\Sigma,\mu)$ that is (globally) convergent in measure has a subsequence that is convergent almost everywhere to the same limit; see \cite[Theorem~19.4]{aliprantis_burkinshaw_PRINCIPLES_OF_REAL_ANALYSIS_THIRD_EDITION:1998} or \cite[Theorem~2.30]{folland_REAL_ANALYSIS_SECOND_EDITION:1999}, for example. Another result in a similar vein, but then for nets, is the following. If~$\mu$ is $\sigma$-finite, and if $(f_\alpha)_{\alpha\in A}$ is a net in $\Ell^0(X,\Sigma,\mu)$ that is (locally) convergent in measure, then there exist indices $\alpha_1\leq\alpha_2\leq\alpha_3\leq\dotsc$ such that the sequence $(f_{\alpha_n})_{n\geq 1}$ converges almost everywhere to the same limit; see \cite[Theorem~7.11]{deng_de_jeu:2022a}. In this case, there is the following way to rephrase the statement: if $(f_\alpha)_{\alpha\in A}$ is a net in $\Ell^0(X,\Sigma,\mu)$ that is convergent in the uo-Lebesgue topology of $\Ell^0(X,\Sigma,\mu)$, then there exists indices $\alpha_1\leq\alpha_2\leq\alpha_3\leq\dotsc$ such that the sequence $(f_{\alpha_n})_{n\geq 1}$ is unbounded order convergent to the same limit. There are quite a few other results in the literature, stating that a net (resp.\ a sequence), which is convergent in a particular locally solid topology on a vector lattice, has such an `embedded sequence' (resp.\ a subsequence) that is unbounded order convergent to the same limit.\footnote{See \cref{def:embedded_sequence} for a precise definition of embedded sequences.} In this paper, we establish a result\footnote{See \cref{res:subsequence}.} which asserts precisely this, but then for a reasonably large class of locally solid topologies all at once. It has many of the known results in this vein as special cases, often in an improved form.

\medskip

\noindent
This paper is organised as follows.

\cref{sec:preliminaries} contains the necessary definitions, notations, conventions, and preliminary results. We elaborate a little on the countable sup property, since this is often a sufficient condition for the applications in the final \cref{sec:applications}.

\cref{sec:metrisable_and_submetrisable_topologies_on_vector_lattices} contains a study of metrisability and submetrisability of locally solid topologies on vector lattices. The topologies to which our key result in \cref{sec:embedded_unbounded_order_convergent_sequences_in_convergent_nets}  applies are what we call the locally interval complete solidly submetrisable topologies,\footnote{See \cref{def:l.c.m.t.}.} and we take some effort to show that topologies of practical interest fall in this category. This is the case, for example, for Fatou topologies on a Dedekind complete vector lattice that have full carriers (which is automatic in the presence of the countable sup property). Due to further properties of Fatou topologies and unbounded order convergence, our result then effectively applies to Fatou topologies with full carriers on arbitrary vector lattices, in particular to Fatou topologies on vector lattices with the countable sup property.\footnote{See \cref{res:subsequence_fatou}.}

\cref{sec:embedded_unbounded_order_convergent_sequences_in_convergent_nets} contains our core result regarding the existence of embedded unbounded order convergent sequences.\footnote{See \cref{res:subsequence}.}

In the final \cref{sec:applications}, we combine the material on locally solid topologies from \cref{sec:metrisable_and_submetrisable_topologies_on_vector_lattices} with the basic theorem from \cref{sec:embedded_unbounded_order_convergent_sequences_in_convergent_nets}, yielding several results  on the existence of embedded unbounded order convergent sequences in topologically convergent nets. For a uo-Lebesgue topology, such an embedded unbounded order convergent sequence is again convergent in this topology, resulting in stronger statements in this case. We also give applications to weak closures and to the equality of adherences and closures related to various convergence structures. During \cref{sec:applications}, we point out how some of our theorems specialise to (improved versions of) known results.

\section{Preliminaries}\label{sec:preliminaries}

\noindent
In this section, we collect the necessary definitions, notations, conventions, and preliminary results.

All vector spaces are over the real numbers.  A linear topology on a vector space is supposed to be Hausdorff. Neighbourhoods of points need not be open.
We recall that a topological vector space is metrisable if and only if zero has a countable neighbourhood basis; see \cite[Theorem~I.6.1]{schaefer_wolff_TOPOLOGICAL_VECTOR_SPACES_SECOND_EDITION:1999}, for example.

\subsection{Vector lattices}\label{subsec:vector_lattices}

 All vector lattices are supposed to be Archimedean. For a vector lattice~$E$, we let $\odual{E}$ denote its order dual, and $\ocdual{E}$ its order continuous dual.

A net $(x_\alpha)_{\alpha\in A}$ in a vector lattice~$E$ is said to be \emph{order convergent in~$E$ to $x\in E$} (denoted by $x_\alpha \oconv x$) if there exists a net $(y_\beta)_{\beta\in B}$ in~$E$ such that $y_\beta\downarrow 0$ and with the property that, for every $\beta_0\in B$, there exists an $\alpha_0\in A$ such that $\abs{x_\alpha - x}\leq y_{\beta_0}$ for all $\alpha\in A$ with  $\alpha\geq \alpha_0$. A net $(x_\alpha)_{\alpha\in A}$ in~$E$ is said to be \emph{unbounded order convergent} (or \emph{uo-convergent} for short) \emph{in~$E$ to $x\in E$} (denoted by $x_\alpha \uoconv x$) if $\abs{x_\alpha - x}\wedge \abs{y}\oconv 0$ in~$E$ for all $y\in E$. Order convergence implies unbounded order convergence to the same limit.

A subset~$S$ of a vector lattice~$E$ is \emph{order closed} if $x\in S$ whenever a net $(x_\alpha)_{\alpha\in A}$ in~$S$ and $x\in E$ are such that $x_\alpha\oconv x$.

Let~$F$ be a vector sublattice of a vector lattice~$E$. We recall that~$F$ is said to be \emph{order dense in~$E$} if, for every $x>0$ in~$E$, there exists a $y\in F$ with $0<y\leq x$; \emph{majorising in~$E$} if, for every $x\in E$, there exists a $y\in F$ with $x\leq y$; and \emph{a regular vector sublattice of~$E$} if $x_\alpha\oconv x$ in~$E$ whenever a net $(x_\alpha)_{\alpha\in A}$ in~$F$ and $x\in F$ are such that $x_\alpha\oconv x$ in~$F$. Ideals and order dense vector sublattices are regular vector sublattices; see \cite[p.653]{gao_troitsky_xanthos:2017}, for example. Furthermore, if~$F$ is a regular vector sublattice of~$E$, $(x_\alpha)_{\alpha\in A}$ is a net in~$F$, and $x\in F$, then $x_\alpha\uoconv x$ in~$F$ if and only if $x_\alpha\uoconv x$ in~$E$. This property even characterises the regular vector sublattices among the vector sublattices; see \cite[Theorem~3.2]{gao_troitsky_xanthos:2017}.

A non-empty subset~$A$ of a vector lattice~$E$ is said to be an \emph{order basis} of~$E$ when the band generated by~$A$ equals~$E$, i.e., when $A^\dc=\left\{0\right\}$; cf.\ \cite[Definition~28.1]{luxemburg_zaanen_RIESZ_SPACES_VOLUME_I:1971}. If~$F$ is an order dense vector sublattice of~$E$, then it is easy to see that an order basis of~$F$ is also an order basis of~$E$.

Let~$E$ be a vector lattice. For $x\in E$, the ideal and the band in~$E$ that are generated by~$x$ are denoted by~$I_x$ and~$B_x$, respectively; the ideal and the band in the Dedekind completion~$E^\delta$ of~$E$ that are generated by~$x$ are denoted by~$I_x^\delta$ and~$B_x^\delta$, respectively. Using the fact that the Dedekind completion of a vector lattice~$E$ is the unique (up to isomorphism) Dedekind complete vector lattice that contains~$E$ as an order dense and majorising vector sublattice, it is easy to see that~$I_x^\delta$ is the Dedekind completion of~$I_x$.

\subsection{The countable sup property}\label{subsec:countable_sup_property}

A vector lattice~$E$ has the \emph{countable sup property} when every subset of~$E$ that has a supremum in~$E$ contains an at most countable subset with the same supremum. In some sources, such as  \cite{luxemburg_zaanen_RIESZ_SPACES_VOLUME_I:1971} and \cite{zaanen_INTRODUCTION_TO_OPERATOR_THEORY_IN_RIESZ_SPACES:1997},~$E$ is then said to be \emph{order separable}. Since it will become apparent in \cref{sec:applications} that having the countable sup property is a condition that is sufficient for several of the results in that section to hold, we use the opportunity to argue that, in practice, quite a few vector lattices have this property. If~$E$ and~$F$ are vector lattices such that there exists a strictly positive linear operator $T\colon E\to F$, then~$E$ has the countable sup property when~$F$ has; see \cite[Theorem~1.45]{aliprantis_burkinshaw_LOCALLY_SOLID_RIESZ_SPACES_WITH_APPLICATIONS_TO_ECONOMICS_SECOND_EDITION:2003}. In particular, if~$E$ has a strictly positive linear functional, then it has the countable sup property. Hence every separable Banach lattice has the countable sup property, as a consequence of  \cite[Exercise~4.1.4]{aliprantis_burkinshaw_POSITIVE_OPERATORS_SPRINGER_REPRINT:2006}. It is easy to see from \cite[Theorem~23.2.(iii)]{luxemburg_zaanen_RIESZ_SPACES_VOLUME_I:1971} that a Banach lattice with an order continuous norm also has the countable sup property. For a measure space $(X,\Sigma,\mu)$ where~$\mu$ is semi-finite, $\Ell^0(X,\Sigma,\mu)$ has the countable sup property if and only if~$\mu$ is $\sigma$-finite; see \cite[Proposition~6.5]{kandic_taylor:2018}. When~$E$ has the countable sup property, then so does every vector sublattice of~$E$; this follows from \cite[Theorem~23.5]{luxemburg_zaanen_RIESZ_SPACES_VOLUME_I:1971}. There is a conditional converse: suppose that~$F$ is a vector sublattice of a vector lattice~$E$, and that~$F$ has the countable sup property. If~$F$ has a countable order basis, and if~$F$ is order dense in~$E$, then~$E$ has the countable sup property. Indeed, by \cite[Theorem~6.2]{kandic_taylor:2018}, the universal completion~$F^{\mathrm u}$ of~$F$ has the countable sub property. Since~$F^{\mathrm u}$ has~$E$ as a vector sublattice by \cite[Theorem~7.23]{aliprantis_burkinshaw_LOCALLY_SOLID_RIESZ_SPACES_WITH_APPLICATIONS_TO_ECONOMICS_SECOND_EDITION:2003},~$E$ also has the countable sup property.

\subsection{Locally solid topologies on vector lattices}\label{subsec:locally_solid_topologies_on_vector_lattices}

A \emph{locally solid} topology on a vector lattice~$E$ is a linear topology on~$E$ such that zero has a neighbourhood basis consisting of solid subsets of~$E$. In this case, $(E,\tau)$ is called a \emph{locally solid vector lattice}. An \emph{o-Lebesgue topology on~$E$} is a locally solid topology~$\tau$ on~$E$ with the property that $x_\alpha\tauconv x$ whenever a net $(x_\alpha)_{\alpha\in A}$ in~$E$ and~$x\in E$ are such that $x_\alpha\oconv x$.\footnote{In the literature, what we call an o-Lebesgue topology is simply called a Lebesgue topology. Now that uo-Lebesgue topologies,  with a completely analogous definition, have become objects of a more extensive study, it seems consistent to also add a prefix to the original term.} A \emph{uo-Lebesgue topology on~$E$} is a locally solid topology~$\tau$ on~$E$ with the property that $x_\alpha\tauconv x$ whenever a net $(x_\alpha)_{\alpha\in A}$ in~$E$ and~$x\in E$ are such that $x_\alpha\uoconv x$. A \emph{Fatou topology on~$E$} is a locally solid topology on~$E$ such that zero has a neighbourhood basis consisting of order closed solid subsets of~$E$. A uo-Lebesgue topology is an o-Lebesgue topology, and an o-Lebesgue topology is a Fatou topology (see \cite[Lemma~4.2]{aliprantis_burkinshaw_LOCALLY_SOLID_RIESZ_SPACES_WITH_APPLICATIONS_TO_ECONOMICS_SECOND_EDITION:2003}).  A vector lattice admits at most one (Hausdorff) uo-Lebesgue topology; see \cite{conradie:2005} or \cite[Theorem~5.5]{taylor:2019}.

Suppose that~$E$ is a vector lattice, and that~$\tau_F$ is a locally solid topology on an order dense ideal~$F$ in~$E$. Take a solid $\tau_F$-neighbourhood of zero in~$F$ and~$y$ in~$F$, and define the solid subset $U_{V,y}$ of~$E$ by setting
\[
U_{V,y}\coloneqq\{x\in E:\abs{x}\wedge \abs{y}\in V\}.
\]
As~$V$ runs through a $\tau_F$-neighbourhood basis at zero that consists of solid subsets of~$F$, and as~$y$ runs over~$F$, the $U_{V,y}$ form a neighbourhood basis at zero for a locally solid topology on~$E$. This topology is called the \emph{unbounded topology on~$E$ that is generated by~$\tau_F$}, and it is denoted by~$\unb_F\tau_F$. We refer to \cite[Theorem~3.1]{deng_de_jeu:2022a} for details on this construction, which refines ideas in \cite{taylor:2019} that are also already implicit in \cite{conradie:2005}. When~$F=E$, where~$E$ has a locally solid topology~$\tau$, we simply write~$\unb\tau$ for $\unb_E\tau$. For a Banach lattice~$E$ with its norm topology~$\tau$, $\unb\tau$ was introduced in \cite{deng_o_brien_troitsky:2017}. The unbounded norm topology $\unb\tau$ on~$E$ is then usually referred to as the un-topology. For a Banach lattice~$F$ with its norm topology~$\tau_F$ which is an order dense ideal in a vector lattice~$E$, $\unb_F\tau_F$ was introduced in \cite{kandic_li_troitsky:2018}. The unbounded topology $\unb_F\tau_F$ on~$E$ is then usually referred to as the un-$F$-topology.

Let $(E,\tau)$ be a locally solid vector lattice, and let $(V_n)_{n\geq 1}$ be a sequence of $\tau$-neighbourhoods of zero. We say that the sequence is \emph{normal} if
 \[
 V_{n+1}+V_{n+1}\subseteq V_n
 \]
 for all~$n$, and we let~$\mathcal{N}$ denote the collection of all normal sequences consisting of solid $\tau$-neighbourhoods of zero. The \emph{carrier of~$\tau$},  denoted by~$C_\tau$, is defined by setting
\[
C_\tau\coloneqq\bigcup \left\{N^\dc: \text{there exists }(V_n)_{n\geq 1} \in \mathcal{N} \text{ such that }N=\bigcap_{n\geq 1} V_n\right\}.
\]
Clearly, if~$\tau$ is metrisable, then~$C_\tau=E$. We shall need the following result in the proof of \cref{res:subsequence_uo-lebesgue_induced}.

\begin{lemma}\label{res:carrier_of_induced_unbounded_topology}
Let~$E$ be a vector lattice, and let~$F$ be an order dense ideal in~$E$. Suppose that~$\tau_F$ is a metrisable locally solid topology on~$F$. Then $F\subseteq C_{u_F\tau_F}$.
 \end{lemma}

\begin{proof}
Choose a normal sequence $(V_n)_{n\geq 1}$ of solid $\tau_F$-neighbourhoods of zero in~$F$ such that $\bigcap_{n\geq 1}V_n=\{0\}$.
For $y\in F$, the sequence $(U_{V_n,y})_{n\geq 1}$ is a sequence of solid  $\unb_F\tau_F$-neighbourhoods of zero. It is easily seen that this is a normal sequence, and also that  $\bigcap_{n\geq 1} U_{V_n,y}=\{y\}^\dc$, where the disjoint complement is taken in~$E$. Hence $y\in C_{\unb_F\tau_F}$.
\end{proof}

\subsection{Embedded sequences in nets}

The main theme of this paper is the existence of unbounded order convergent sequences that are embedded in topologically convergent nets in the sense of the following definition.

\begin{definition}\label{def:embedded_sequence}
	Let~$S$ be a non-empty set, let~$A$ be a directed set, and let $(x_\alpha)_{\alpha\in A}$ be a net in~$S$. Suppose that $\alpha_1,\alpha_2,\alpha_3,\dotsc$ in~$A$ are such that $\alpha_1\leq\alpha_2\leq\alpha_3,\dotsc$, and such that $\alpha_1<\alpha_2<\alpha_3,\dotsc$ when~$A$ has no largest element. Then the sequence $(x_{\alpha_n})_{n\geq 1}$ is said to be \emph{embedded in the net $(x_\alpha)_{\alpha\in A}$}.
\end{definition}

It is not required that the~$\alpha_n$ form a cofinal subset of~$A$, only that the map $n\mapsto\alpha_n$ be increasing, and strictly increasing when~$A$ has no largest element. The embedded sequences in a given sequence are its subsequences. 

\cref{res:subsequence} asserts the existence of embedded unbounded order convergent sequences in topologically convergent nets, where the indices~$\alpha_n$ for the sequences can even be found in a specific way. A particular case is the existence of unbounded order convergent subsequences of topologically convergent sequences. When the pertinent topology is metrisable, these two properties are equivalent. We now proceed to show this.

\begin{lemma}\label{res:nets_and_sequences_no_largest_element}
	Let $(E,\tau)$ be a topological vector lattice, where~$\tau$ is metrisable. The following are equivalent:
	
	\begin{enumerate}
		\item\label{part:nets_and_sequences_no_largest_element_1}
		for every directed set~$A$ that has no largest element, for every net $(x_
		\alpha)_{\alpha\in A}$ and~$x$ in~$E$ such that $x_\alpha\tauconv x$, one can choose any $\widetilde\alpha_1\in A$; then find an $\alpha_1\in A$; then choose any $\widetilde\alpha_2\in A$; then find an $\alpha_2\in A$; etc., such that:
		\begin{enumerate}
			
			\item
			$\alpha_1<\alpha_2<\alpha_3<\dotsc$;
			
			\item
			$\alpha_n>\widetilde\alpha_n$ for $n\geq 1$;
			
			\item
			$x_{\alpha_n}\uoconv x$ as $n\to\infty$.
			
		\end{enumerate}
		\item\label{part:nets_and_sequences_no_largest_element_2}
		every sequence in~$E$ that is $\tau$-convergent to an element~$x$ of~$E$ has a subsequence that is uo-convergent to~$x$.
	\end{enumerate}
\end{lemma}

\begin{proof}
	It is clear that \ref{part:nets_and_sequences_no_largest_element_1} implies \ref{part:nets_and_sequences_no_largest_element_2}. We prove that \ref{part:nets_and_sequences_no_largest_element_2} implies \ref{part:nets_and_sequences_no_largest_element_1}. Suppose that an index set~$A$ with no largest element, a net $(x_
	\alpha)_{\alpha\in A}$ in~$E$, and $x\in E$ are such that $x_\alpha\tauconv x$. Take a countable $\tau$-neighbourhood basis $V_1,V_2,V_3,\dotsc$ of~$x$.  Choose any $\widetilde\alpha_1\in A$. 
	  Since~$A$ has no largest element, there exists an index $\alpha_1^\prime>\widetilde\alpha_1$ such that $x_\alpha\in V_1$ whenever $\alpha\geq\alpha_1^\prime$.   
	  For $k\geq 2$, we inductively choose any $\widetilde\alpha_k$ in~$A$, and then find an index $\alpha_k^\prime$ with $\alpha_k^\prime>\widetilde\alpha_{k}$ and $\alpha_k^\prime>\alpha_{k-1}^\prime$, such that $x_\alpha\in V_k$ whenever $\alpha\geq\alpha_k^\prime$. Then $x_{\alpha_k^\prime}\tauconv x$ as $k\to\infty$. By assumption, there is a subsequence $(x_{\alpha_{k_n}^\prime})_{n\geq 1}$ such that  $x_{\alpha_{k_n}^\prime}\uoconv x$ as $n\to\infty$. For $n\geq 1$, set $\alpha_n\coloneqq \alpha_{k_n}^\prime$. Then $\alpha_1,\alpha_2,\alpha_3,\dotsc$ are as required.
\end{proof}

When the index set~$A$ of a convergent net $(x_\alpha)_{\alpha\in A}$ has a largest element $\alpha_{\mathrm{largest}}$, its limit equals $x_{\alpha_{\mathrm{largest}}}$. Taking this into account, we have the following less precise consequence of \cref{res:nets_and_sequences_no_largest_element}.

\begin{corollary}\label{res:nets_and_sequences_general}
	Let $(E,\tau)$ be a topological vector lattice, where~$\tau$ is metrisable. The following are equivalent:
	\begin{enumerate}
		\item\label{part:nets_and_sequences_general_1}
		for every net $(x_\alpha)_{\alpha\in A}$ and~$x$ in~$E$ such that $x_\alpha\tauconv x$, there exist indices such that $\alpha_1\leq\alpha_2\leq\alpha_3,\dotsc$, and such that $\alpha_1<\alpha_2<\alpha_3<\dotsc$ when~$A$ has no largest element, with $x_{\alpha_n}\uoconv x$ as $n\to\infty$
		
		\item\label{part:nets_and_sequences_general_2}
		every sequence in~$E$ that is $\tau$-convergent to an element~$x$ of~$E$ has a subsequence that is uo-convergent to~$x$.
	\end{enumerate}
\end{corollary}

Obviously, there is a more general principle that underlies \cref{res:nets_and_sequences_general}, as its proof merely exploits the countability of a neighbourhood basis. We refrain from attempting to phrase this, but we do note that its analogues for, e.g., order convergence or relative uniform convergence are clearly also true.

\subsection{Measure spaces}

When $(X,\Sigma, \mu)$ is a measure space, we let $\Ell^0(X,\Sigma,\mu)$ denote the vector lattice of measurable functions on~$X$, with identification of two functions that are $\mu$-almost everywhere equal. Let $(f_\alpha)_{\alpha\in A}$ be a net in $\Ell_0(X,\Sigma,\mu)$, and let $f\in \Ell^0(X,\Sigma,\mu)$. Then $(f_\alpha)_{\alpha\in A}$ is said to \emph{converge globally in measure to~$f$} if
\[
\lim_\alpha \mu\big(\left\{t\in X: \abs{f_\alpha(t)-f(t)}\geq \varepsilon\right\}\big)=0
\]
for every $\varepsilon>0$, and to \emph{converge locally in measure to~$f$} (denoted by  $f_\alpha\conv{\mu^\ast} f$) if
\[
\lim_\alpha \mu\big(\left\{t\in A: \abs{f_\alpha(t)-f(t)}\geq \varepsilon\right\}\big)=0
\]
for every $\varepsilon>0$ and every measurable subset~$A$ of~$X$ with finite measure.

We recall that a measure space $(X,\Sigma,\mu)$ is called \emph{semi-finite} when, for every $A\in \Sigma$ with $\mu(A)=\infty$, there exists a measurable subset~$A_0$ of~$A$ such that $0<\mu(A_0)<\infty$. Clearly, $\sigma$-finite measure spaces are semi-finite.

\section{Metrisable and locally solidly submetrisable locally solid topologies}\label{sec:metrisable_and_submetrisable_topologies_on_vector_lattices}

\noindent
This section starts with a general metrisation theorem for locally solid vector lattices in \cref{subsec:metrisable_locally_solid_topologies}. It is of some interest in its own right, and essential to a number of proofs in the remainder of \cref{sec:metrisable_and_submetrisable_topologies_on_vector_lattices}. In \cref{subsec:locally_solidly_metrisable_locally_solid_topologies}, we consider locally solid submetrisability of general locally solid topologies. The final \cref{subsec:norm_unbounded_norm_lebesgue_and_fatou} specialises to norm, unbounded norm, o-Lebesgue,  and Fatou topologies.

\subsection{Metrisable locally solid topologies}\label{subsec:metrisable_locally_solid_topologies}

Suppose that a locally solid topology on a vector lattice is metrisable. In this case, a compatible metric can be chosen with convenient properties, as is shown by the following result. We shall need only the case of locally solid topologies, but we also include the case of locally convex-solid topologies for the sake of completeness.

\begin{theorem}\label{res:metrisation_theorem}
	
	Let $(E,\tau)$ be a locally solid vector lattice. If~$\tau$ is metrisable, then there exists a metric $d\colon E\times E\to[0,1]$ on~$E$ with the following properties:
	\begin{enumerate}
		\item\label{part:metrisation_theorem_1}
		$d$ is compatible with~$\tau$;
		
		\item\label{part:metrisation_theorem_2}
		$d$ is translation invariant;
		
		\item\label{part:metrisation_theorem_3}
		$d\left(0,x\right)\leq d\left(0,y\right)$ for $x,y\in E$ such that $\abs{x}\leq\abs{y}$;
		
		\item\label{part:metrisation_theorem_4}
		the metric balls $\left\{x\in E: d(0,x)<r\right\}$ and $\left\{x\in E: d(0,x)\leq r\right\}$ are solid for all $r\geq 0$.
	\end{enumerate}
	If, in addition,~$\tau$ is locally convex, then there exists a metric $d\colon E\times E\to[0,1]$ on~$E$ satisfying  \ref{part:metrisation_theorem_1}--\ref{part:metrisation_theorem_4}, as well as:
	\begin{enumerate}[resume]
		\item\label{part:metrisation_theorem_5}
		the metric balls $\left\{x\in E: d(c,x)<r\right\}$ and $\left\{x\in E: d(c,x)\leq r\right\}$ are convex for all $c\in E$ and $r\geq 0$.
	\end{enumerate}
\end{theorem}

\begin{proof}
	We recall how a metric on a metrisable topological vector space~$E$ can be found that satisfies~\ref{part:metrisation_theorem_1} and~\ref{part:metrisation_theorem_2}. We start by choosing a neighbourhood basis $(V_n)_{n\geq 1}$ of (not necessarily open) balanced neighbourhoods of zero such that
	\begin{equation}\label{eq:local_base_inclusions}
		V_{n+1}+V_{n+1}\subseteq V_n
	\end{equation}
	for $n\geq 1$. For each non-empty finite subset~$H$ of $\left\{1,2,\dotsc\right\}$, we set $V_H\coloneqq \sum_{n\in H} V_n$ and $p_H\coloneqq\sum_{n\in H}2^{-n}$. For $x\in E$, set
	\begin{equation*}
		\abs{x}\coloneqq
		\begin{cases}
			1&\text{ if } x \text{ is not in any of the } V_H;\\
			\inf_H \left\{p_H: x\in V_H\right\}&\text{ otherwise}.
		\end{cases}	
	\end{equation*}
	As is shown in \cite[Proof of Theorem~I.6.1]{schaefer_wolff_TOPOLOGICAL_VECTOR_SPACES_SECOND_EDITION:1999}, one then obtains a metric~$d$ on~$E$ satisfying~\ref{part:metrisation_theorem_1} and~\ref{part:metrisation_theorem_2} by defining $d\left(x,y\right)\coloneqq \abs{x-y}$ for $x,y\in E$.

	In the case of a locally solid vector lattice, one can choose the~$V_n$ such that they are also solid. Since a finite sum of solid subsets is solid as a consequence of (a precise form of) the Riesz decomposition (see \cite[Theorem~1.10]{aliprantis_burkinshaw_LOCALLY_SOLID_RIESZ_SPACES_WITH_APPLICATIONS_TO_ECONOMICS_SECOND_EDITION:2003}), it is then clear that~\ref{part:metrisation_theorem_3} holds, which obviously implies~\ref{part:metrisation_theorem_4}.
	
	Suppose that~$\tau$ is locally solid as well as locally convex. In this case, the~$V_n$ can be chosen to be solid as well as convex, as a consequence of the fact that the convex hull of a solid set is solid (see \cite[Theorem~1.11]{aliprantis_burkinshaw_LOCALLY_SOLID_RIESZ_SPACES_WITH_APPLICATIONS_TO_ECONOMICS_SECOND_EDITION:2003}). The resulting metric~$d$ from the construction satisfies \ref{part:metrisation_theorem_1}--\ref{part:metrisation_theorem_4}; we now show that the convexity of the~$V_n$ yields~\ref{part:metrisation_theorem_5}. By the  translation invariance of~$d$, it suffices to take $c=0$ in~\ref{part:metrisation_theorem_5}. Since the~$V_n$ are convex, so are the~$V_H$. Furthermore, if $H_1, H_2$ are nonempty subsets of $\left\{1,2,\dotsc\right\}$ such that $p_{H_1}\leq p_{H_2}$, then a moment's thought shows that $V_{H_1}\subseteq V_{H_2}$ as a consequence of \cref{eq:local_base_inclusions}. Together with the density of the dyadic rationals, this implies that
	\begin{equation*}
		\left\{x\in E: d\left(0,x\right)\leq r\right\}=
		\begin{cases} \bigcap_{H:p_H>r} V_H&\text{ if }0\leq r< 1;\\
			E&\text{ if }r\geq 1;\\
		\end{cases}
	\end{equation*}
	which is convex. Since $\left\{x\in E: d\left(0,x\right)<r\right\}=\bigcup_{0\leq s<r }\left\{x\in E: d\left(0,x\right)\leq s\right\}$ is a nested union of convex sets, all open balls are convex as well.
\end{proof}

\subsection{Locally solidly submetrisable locally solid topologies}\label{subsec:locally_solidly_metrisable_locally_solid_topologies}

We recall that a linear topology on a topological vector space is called \emph{submetrisable} if it is finer than a metrisable linear topology on the space. In \cite{kandic_taylor:2018}, metrisability and submetrisability of minimal and of unbounded locally solid topologies are studied. In this section, we are mainly concerned with a \emph{local} version of submetrisability, and then for general locally solid topologies.

\begin{definition}\label{def:l.c.m.t.}
	Let $(E,\tau)$ be a locally solid vector lattice.  We say that~$\tau$ is:
	\begin{enumerate}
		\item\label{part:l.c.m.t._1}
		\emph{\sumet}\ if it is finer than some metrisable locally solid topology on~$E$;\footnote{In \cite{kandic_taylor:2018},~$\tau$ is then called \emph{Riesz submetrisable}. The present terminology may be a little more suggestive.}
		
		\item\label{part:l.c.m.t._2}
		\emph{\locsumet}\ if, for every $x\in E$, the restricted topology  $\tau\vert_{B_x}$ is \sumet;
		
		\item\label{part:l.c.m.t._3}
		\emph{\locsucommet}\ if, for every $x\in E^+$, there exists a metrisable locally solid topology $\widetilde\tau_x$ on~$B_x$ such that:
		\begin{enumerate}
			\item\label{part:l.c.m.t._3_a}
			 the restricted topology $\tau\vert_{B_x}$ is finer than $\widetilde\tau_x$;
			
			\item\label{part:l.c.m.t._3_b}
			every $\widetilde\tau_x$-Cauchy sequence in the order interval $[0,x]$ is $\widetilde\tau_x$-conver\-gent to an element of $[0,x]$.
	\end{enumerate}
	\end{enumerate}  	
\end{definition}

Clearly, when a locally solid topology is \sumet\ or \locsucommet, it is \locsumet. When the vector lattice has a weak order unit, being \sumet\ and being \locsumet\ are equivalent. The locally solid topologies to which the basic \cref{res:subsequence} applies are the \locsucommet\ ones.

\begin{remark}\label{rem:invariant}\quad
	\begin{enumerate}		
		\item
		It follows from \cref{res:metrisation_theorem} that the metrics figuring in the three parts of \cref{def:l.c.m.t.} can be chosen to be translation invariant and such that they are `lattice metrics' as in part~\ref{part:metrisation_theorem_3} of that theorem. In the case of a \locsucommet\ topology~$\tau$, they then make $[0,x]$ into a complete metric space as a consequence of their translation invariance; see \cite[p.20-21]{rudin_FUNCTIONAL_ANALYSIS_SECOND_EDITION:1991}. These observations will frequently be used in the sequel.
		
		\item
		The completeness in part~\ref{part:l.c.m.t._3} is required only for the single interval $[0,x]$, but we have refrained from including this into the terminology. This minimal requirement is already sufficient for the proof of the key \cref{res:subsequence} to be valid. In fact, as the case of the un-topology will show, a more `natural' requirement, such as requiring that all order intervals in~$B_x$ be $\widetilde\tau_x$-complete, could well be asking too much; see \cref{rem:proof_for_un-topology}.
	\end{enumerate}
\end{remark}

For a locally solid topology, being metrisable and being solidly submetrisable are actually equivalent, as is shown by our next result, \cref{res:redundancy}.\footnote{In the terminology of \cite{kandic_taylor:2018}, it shows that, for locally solid topologies, submetrisability coincides with Riesz submetrisability.} Hence there is a certain redundancy in the parts~\ref{part:l.c.m.t._1} and~\ref{part:l.c.m.t._2} of \cref{def:l.c.m.t.}, where we could equally well have spoken of `submetrisable' and `locally submetrisable'. This redundancy is, however, not present in part~\ref{part:l.c.m.t._3}, since completeness may be lost when passing to a finer topology, as is done in \cref{res:redundancy}. For reasons of uniformity, we have, therefore, still included the solidness in the terminology in all parts of \cref{def:l.c.m.t.}.

\begin{proposition}\label{res:redundancy}
	Let $(E,\tau)$ be a locally solid vector lattice. Suppose that there exists a metrisable linear topology~$\tau^\ast$ on~$E$ that is coarser than~$\tau$. Then there exists a metrisable locally solid topology~$\widetilde\tau$ on~$E$ that is coarser than~$\tau$ and finer than~$\tau^\ast$. If~$\tau$ is an o-Lebesgue topology \uppars{resp.\ a uo-Lebesgue topology}, then any such~$\widetilde\tau$ is an o-Lebesgue topology \uppars{resp.\ a uo-Lebesgue topology}. If~$\tau$ is a Fatou topology, then~$\widetilde\tau$ can be chosen to be a Fatou topology.
\end{proposition}

\begin{proof}
	Take a metrisable linear topology~$\tau^\ast$ on~$E$ that is coarser than~$\tau$, and let $V_1,V_2,\dotsc$ be a countable neighbourhood basis at zero for~$\tau^\ast$. Since~$\tau^\ast$ is coarser than~$\tau$ and~$\tau$ is locally solid, we can choose solid $\tau$-neighbourhoods $U_1,U_2,\dotsc$ of zero such that $U_n\subseteq V_n$ for all~$n$. If~$\tau$ is a Fatou topology, then we choose such~$U_n$ that are also order closed. Set $\widetilde V_1\coloneqq U_1\subseteq V_1$. Then~$\widetilde V_1$ is a solid~$\tau$-neigbourhood of zero, which is order closed if~$\tau$ is a Fatou topology. There exists a solid~$\tau$-neighbourhood~$A_2$ of zero such that $A_2 + A_2\subseteq \widetilde V_1$. If~$\tau$ is a Fatou topology, then we choose such an~$A_2$ that is order closed. Set $\widetilde V_2\coloneqq A_2\cap U_2\subseteq U_2\subseteq V_2$. Then~$\widetilde V_2$ is a solid $\tau$-neighbourhood of zero, which is order closed if~$\tau$ is a Fatou topology. Furthermore, $\widetilde V_2+\widetilde V_2\subseteq A_2+A_2\subseteq \widetilde V_1$. Continuing this way, we construct a normal sequence $\widetilde V_1,\widetilde V_2,\dotsc$ of solid $\tau$-neighbourhoods of zero such that $\widetilde V_n\subseteq V_n$, and which can be chosen to be order closed if~$\tau$ is a Fatou topology. Since $\left\{0\right\}\subseteq\bigcap_{n\geq 1}\widetilde V_n\subseteq \bigcap_{n\geq 1} V_n=\left\{0\right\}$, it is now clear from \cite[Theorem~5.1]{kelley_namioka_LINEAR_TOPOLOGICAL_SPACES_SECOND_CORRECTED_PRINTING:1976} that the~$\widetilde V_n$ form a neighbourhood basis at zero for a linear topology on~$E$ which is obviously coarser than~$\tau$, finer than~$\tau^\ast$, metrisable, locally solid, and a Fatou topology if~$\tau$ is.
	
	The statement that~$\widetilde\tau$ is an o-Lebesgue topology or a uo-Lebesgue topology when~$\tau$ is, is obviously already true for any locally solid topology that is coarser than~$\tau$.
\end{proof}

The following is an easy consequence of \cref{res:redundancy} and the fact that the inclusion map from a regular vector sublattice into the superlattice preserves order convergence and unbounded order convergence. It will be used in the proof of \cref{res:when_fatou_is_submetrisable_dedekind}.

\begin{corollary}\label{res:can_be_lebesgue_fatou}
	Let $(E,\tau)$ be a locally solid vector lattice, and let~$F$ be a regular vector sublattice of~$E$ such that  there exists a metrisable linear topology~$\tau^\ast$ on~$F$ that is coarser than~$\tau\vert_F$. Then there exists a metrisable locally solid topology~$\widetilde\tau$ on~$F$ that is coarser than~$\tau\vert_F$ and finer than~$\tau^\ast$. If~$\tau$ is an o-Lebesgue topology \uppars{resp.\ a uo-Lebesgue topology}, then~$\tau\vert_F$ and any such~$\widetilde\tau$ are  o-Lebesgue topologies \uppars{resp.\ uo-Lebesgue topologies}. If~$\tau$ is a Fatou topology, then so is~$\tau\vert_F$, and in this case~$\widetilde\tau$ can be chosen to be a Fatou topology.
\end{corollary}

In the definition of local submetrisability, we could also have used a seemingly weaker condition, as is shown by the following result that will be used in the proof of \cref{res:when_fatou_is_submetrisable_dedekind}.

\begin{proposition}\label{res:tau_I_x_B_x}
	Let $(E,\tau)$ be a locally solid vector lattice. Take $x\in E$. The following are equivalent:
	\begin{enumerate}
		\item\label{part:tau_I_x_B_x_1}
		$\tau\vert_{B_x}$ is \sumet;
		
		\item\label{part:tau_I_x_B_x_2}
		$\tau\vert_{I_x}$ is \sumet.
	\end{enumerate}
\end{proposition}

\begin{proof} It is clear that~\ref{part:tau_I_x_B_x_1} implies~\ref{part:tau_I_x_B_x_2}. For the converse, suppose that there exists a locally solid topology~$\tau_x^\ast$ on~$I_x$ that is coarser than $\tau\vert_{I_x}$, and which is metrisable by means of a metric~$d_x^\ast$. By \cref{res:metrisation_theorem}, we may suppose that~$d_x^\ast$ is translation invariant, and that  $d_x^\ast\left(0,y_1\right)\leq d_x^\ast\left(0,y_2\right)$ for $y_1,y_2\in I_x$ with $\abs{y_1}\leq\abs{y_2}$.   Define $\widetilde d_x\colon  B_x\times B_x\to[0,\infty)$ by setting $\widetilde{d}_x\left(z_1, z_2\right)\coloneqq d_x^\ast\left(0, \abs{z_1-z_2}\wedge \abs{x}\right)$ for $z_1,z_2\in B_x$. Arguing as in the proof of \cite[Theorem~3.2]{kandic_marabeh_troitsky:2017}, it is easily seen that~$\widetilde{d}_x$ is a translation invariant metric on~$B_x$, and that the metric balls $\left\{z\in B_x:\widetilde d_x\left(0,z\right)<r\right\}$ are solid subsets of~$B_x$ for $r>0$.  Take $z\in B_x$.  Then $\abs{\lambda z}\conv{\tau} 0$ as $\lambda\to 0$. Since~$\tau$ is locally solid, we also have that  $\abs{\lambda z}\wedge\abs{x}\conv{\tau}0$. Hence $\abs{\lambda z}\wedge\abs{x}\conv{\tau\vert_{I_x}}0$, and then also $\abs{\lambda z}\wedge\abs{x}\conv{\tau_x^\ast}0$. This shows that $\widetilde d_x\left(0,\lambda z\right)\to 0$ as $\lambda\to 0$, so that the open balls  $\left\{z\in B_x:\widetilde d_x\left(0,z\right)<r\right\}$ are absorbing subsets of~$B_x$. It is now easy to see that these open balls satisfy the conditions in \cite[Theorem~5.1]{kelley_namioka_LINEAR_TOPOLOGICAL_SPACES_SECOND_CORRECTED_PRINTING:1976} to be a neighbourhood basis at zero for a linear topology~$\widetilde\tau_x$ on~$B_x$, which is obviously locally solid.

We are left to show that~$\widetilde\tau_x$ is coarser than $\tau\vert_{B_x}$. Let $(x_\alpha)_{\alpha\in A}$ be a net in~$B_x$ such that $x_\alpha\conv{\tau\vert_{B_x}} 0$ . Then $x_\alpha\conv{\tau} 0$ in~$E$, so that  $\abs{x_\alpha}\wedge\abs{x}\conv{\tau} 0$ in~$E$ as~$\tau$ is locally solid. Hence $\abs{x_\alpha}\wedge\abs{x}\conv{\tau\vert_{I_x}} 0$, and then also $\abs{x_\alpha}\wedge\abs{x}\conv{\tau_x^\ast} 0$ in~$I_x$. The definition of~$\widetilde{d}_x$ then shows that $x_\alpha\conv{\widetilde \tau_x} 0$ in~$B_x$, as desired.
	\end{proof}

For a locally solid topology, being \locsumet\ is determined by its carrier. 

\begin{proposition}\label{res:when_is_submetrisable}
	Let $(E,\tau)$ be a locally solid vector lattice. The following are equivalent:
	\begin{enumerate}
		\item\label{part:when_is_submetrisable_1}
		$\tau$ is \locsumet;
		\item\label{part:when_is_submetrisable_2}
		$C_\tau=E$.
	\end{enumerate}
\end{proposition}

\begin{proof}
	We prove that \ref{part:when_is_submetrisable_1}~$\Rightarrow$~\ref{part:when_is_submetrisable_2}. Take $x\in E$. Using \cref{res:metrisation_theorem}, we see that there exists a translation invariant metric $d_x$ on $B_x$ that defines a topology on $B_x$ coarser than $\tau\vert_{B_x}$, and such that $d_x(0,y_1)\leq d_x(0,y_2)$ whenever $y_1,y_2\in B_x$ are such that $\abs{y_1}\leq\abs{y_2}$. Define $\rho_x\colon E\to{\mathbb R}$ by setting $\rho_x(z)\coloneqq d_x(0,\abs{x}\wedge\abs{z})$ for $z\in E$. Then $\rho_x$ is a $\tau$-continuous Riesz pseudonorm on~$E$ in the sense of \cite[Definition~2.27]{aliprantis_burkinshaw_LOCALLY_SOLID_RIESZ_SPACES_WITH_APPLICATIONS_TO_ECONOMICS_SECOND_EDITION:2003}. Its null ideal $I_{\rho_x}$, defined as $I_{\rho_x}\coloneqq\{z\in E: \rho_x(z)=0\}$, is clearly equal to $\{x\}^\dc$. Hence $x\in I_{\rho_x}^\dc$. By \cite[Exercise 4.16]{aliprantis_burkinshaw_LOCALLY_SOLID_RIESZ_SPACES_WITH_APPLICATIONS_TO_ECONOMICS_SECOND_EDITION:2003}, this implies that $x\in C_\tau$.
	
	We prove that \ref{part:when_is_submetrisable_2}~$\Rightarrow$~\ref{part:when_is_submetrisable_1}. Take $x\in E=C_\tau$. By \cite[Exercise~4.16]{aliprantis_burkinshaw_LOCALLY_SOLID_RIESZ_SPACES_WITH_APPLICATIONS_TO_ECONOMICS_SECOND_EDITION:2003}, there exists a $\tau$-continuous Riesz pseudonorm $\rho_x$ on $E$ such that $x\in I_{\rho_x}^\dc$. Define $d_x:B_x\times B_x\to{\mathbb R}$ by setting $d_x(y_1,y_2)\coloneqq \rho_x(\abs{y_1-y_2})$ for $y_1,y_2\in B_x$. Then $d_x$ is a metric on $B_x$ which defines a locally solid topology on $B_x$ that is coarser than $\tau\vert_{B_x}$.  Hence $\tau$ is \locsumet.		
\end{proof}

We conclude this section with a small result which is an immediate consequence of the fact that every order interval in a locally solid vector lattice is closed in the topology.

\begin{proposition}\label{res:metrisable_is_OK}
	Let $(E,\tau)$ be a locally solid vector lattice such that~$\tau$ is metrisable and complete. Then~$\tau$ is \locsucommet
\end{proposition}	

\subsection{Norm, unbounded norm, o-Lebesgue, and Fatou topologies}\label{subsec:norm_unbounded_norm_lebesgue_and_fatou}

In this section, we consider a number of particular locally solid topologies and study when they are locally (interval complete) solidly submetrisable. We start with the norm and unbounded norm topology on a Banach lattice.

\begin{proposition}\label{res:un_is_l.c.sub}
	Let~$E$ be a Banach lattice. Then the norm topology and the unbounded norm topology on~$E$ are \locsucommet.
\end{proposition}

\begin{proof}
	We denote the norm topology on~$E$ by~$\tau$. It is clear from \cref{res:metrisable_is_OK} that~$\tau$ is \locsucommet. We turn to the locally solid unbounded norm topology~$\unb\tau$. Take $x\in\pos{E}$, and define $d_x\colon B_x\times B_x\to[0,\infty)$ by setting
	\[
	d_x\left(x_1,x_2\right)\coloneqq\left\Vert\abs{x_1-x_2}\wedge (2x)\right\Vert
	\]
	for $x_1, x_2\in B_x$. As in the proof of \cite[Theorem~3.2]{kandic_marabeh_troitsky:2017}, it is easily seen that~$d_x$ is a metric on~$B_x$. The facts that~$d_x$ is translation invariant and that the open balls $\left\{y\in B_x: d_x\left(0,y\right)<r\right\}$ for $r>0$ are solid make it simple to verify that these open balls verify the conditions in \cite[Theorem~5.1]{kelley_namioka_LINEAR_TOPOLOGICAL_SPACES_SECOND_CORRECTED_PRINTING:1976} to be a neighbourhood basis at zero for a linear topology~$\widetilde\tau_x$ on~$B_x$, which is then obviously locally solid. It is immediate from the definition of the $\unb\tau$-topology that $\unb\tau\vert_{B_x}$ is finer than~$\widetilde\tau_x$. Finally, the fact that $d_x\left(x_1,x_2\right)=\norm{x_1-x_2}$ for $x_1,x_2\in[0,x]$ makes it clear that $([0,x],d_x)$ is a complete metric space.
\end{proof}

\begin{remark}\label{rem:proof_for_un-topology}
	The proof for the unbounded norm topology can easily be modified to yield a more general result: for $x\in\pos{E}$ and $y,z\in I_x$ with $y\leq z$, there exists a metric $d_{y,z}$ on~$B_x$ that is coarser than~$\tau\vert_{B_x}$ and such that $([y,z], d_{y,z})$ is a complete metric space. Indeed, after choosing $\lambda\geq 0$ such that $[y,z]\subseteq[-\lambda x,\lambda x]$, defining $d_{y,z}\left(x_1,x_2\right)\coloneqq\left\Vert\abs{x_1-x_2}\wedge (2\lambda x)\right\Vert$ for $x_1,x_2\in B_x$ yields such a metric. There does not appear to be an obvious way (if it is at all possible) to produce a metric on~$B_x$ that induces a locally solid topology on~$B_x$ which is coarser than $\unb\tau\vert_{B_x}$, and such that \emph{all} order intervals in~$I_x$ are complete metric spaces, let alone all order intervals in~$B_x$.
\end{remark}

We now turn to o-Lebesgue and Fatou topologies.

If $(E,\tau)$ is a locally solid vector lattice, where~$\tau$ is a Fatou topology, then we know from   \cite[Theorem~4.12]{aliprantis_burkinshaw_LOCALLY_SOLID_RIESZ_SPACES_WITH_APPLICATIONS_TO_ECONOMICS_SECOND_EDITION:2003} that there exists a unique Fatou topology~$\tau^\delta$ on the Dedekind completion~$E^\delta$ of~$E$ extending~$\tau$. When~$\tau$ is an o-Lebesgue topology, so is~$\tau^\delta$.  An inspection of the proof shows that~$\tau^\delta$ is metrisable when~$\tau$ is. Furthermore, if~$\tau_1$ and~$\tau_2$ are Fatou topologies on~$E$ such that~$\tau_1\subseteq\tau_2$, then $\tau_1^\delta\subseteq\tau_2^\delta$. These two observations are used in the proof of the following.

\begin{theorem}\label{res:when_fatou_is_submetrisable_dedekind}
	Let $(E,\tau)$ be a locally solid vector lattice, where~$\tau$ is a Fatou topology. Let~$\tau^\delta$ be the extension of~$\tau$ to a Fatou topology on the Dedekind completion~$E^\delta$ of~$E$. The following are equivalent:
	\begin{enumerate}
		\item\label{dedekind1}
		$\tau$ is \locsumet;
		
		\item\label{dedekind4}
		$\tau^\delta$ is \locsumet;
		
		\item\label{dedekind5}
		$\tau^\delta$ is \locsucommet;
		
		\item\label{dedekind2}
		$C_\tau=E$;
		
		\item\label{dedekind3}
		$C_{\tau^\delta}=E^\delta$.
	\end{enumerate}
When~$\tau$ is an o-Lebesgue topology, then \ref{dedekind1}--\ref{dedekind3} are also equivalent to:
\begin{enumerate}[resume]
	
	\item\label{dedekind6}
	$E$ has the countable sup property.
	
\end{enumerate}
\end{theorem}

\begin{proof}
	We see from \cref{res:when_is_submetrisable} that \ref{dedekind1}~$\Leftrightarrow$~\ref{dedekind2} and that \ref{dedekind4}~$\Leftrightarrow$~\ref{dedekind3}. It is clear that \ref{dedekind5}~$\Rightarrow$~\ref{dedekind4}. The combination of \cref{res:can_be_lebesgue_fatou} and  \cite[Theorem~4.28]{aliprantis_burkinshaw_LOCALLY_SOLID_RIESZ_SPACES_WITH_APPLICATIONS_TO_ECONOMICS_SECOND_EDITION:2003} shows that \ref{dedekind4}~$\Rightarrow$~\ref{dedekind5}. The proof of the equivalence of the first five properties will be complete once we show that \ref{dedekind1}~$\Rightarrow$~\ref{dedekind4} and that \ref{dedekind3}~$\Rightarrow$~\ref{dedekind2}.
	
	We prove that \ref{dedekind1}~$\Rightarrow$~\ref{dedekind4}.
	The first step is to show that $\tau^\delta|_{I_x^\delta}$ is \sumet\ for $x\in E$, as follows. We see from \cref{res:can_be_lebesgue_fatou} that there exists a metrisable Fatou topology~$\widetilde\tau_x$ on~$I_x$ such that $\widetilde\tau_x\subseteq \tau|_{I_x}$. Since~$I_x^\delta$ is the Dedekind completion of~$I_x$, there exists a unique Fatou topology $\left(\widetilde\tau_x\right)^\delta$ (resp.\ $\left(\tau|_{I_x}\right)^\delta$) on~$I_x^\delta$ that extends~$\widetilde\tau_x$ (resp.~$\tau|_{I_x}$). From the observations prior to the  theorem, we see that $\left(\widetilde\tau_x\right)^\delta$ is metrisable and that $\left(\widetilde\tau_x\right)^\delta\subseteq \left(\tau|_{I_x}\right)^\delta$. Now we note that $\tau^\delta|_{I_x^\delta}$ is also a Fatou topology on~$I_x^\delta$, and that $\left(\tau^\delta|_{I_x^\delta}\right)|_{I_x}=\tau^\delta|_{I_x}=\left(\tau^\delta|_E\right)|_{I_x}=\tau|_{I_x}$. Since there is only one Fatou topology on~$I_x^\delta$ extending~$\tau|_{I_x}$, we conclude that $\left(\tau|_{I_x}\right)^\delta=\tau^\delta|_{I_x^\delta}$. Hence $\left(\widetilde\tau_x\right)^\delta\subseteq \tau^\delta|_{I_x^\delta}$, which shows that $\tau^\delta|_{I_x^\delta}$ is \sumet, as desired.
	For the second step, take $x\in E^\delta$. Since~$E$ is majorising in~$E^\delta$, there exists an $y\in E$ such that $I_x^\delta\subseteq I_y^\delta$. As we know that $\tau^\delta|_{I_y^\delta}$ is \sumet, the same is then clearly true for $\tau^\delta_{I_x^\delta}$. \cref{res:tau_I_x_B_x} then yields that~$\tau^\delta$ is \sumet.
	
	We prove that \ref{dedekind3}~$\Rightarrow$~\ref{dedekind2}.  Using \cite[Exercise~4.5]{aliprantis_burkinshaw_LOCALLY_SOLID_RIESZ_SPACES_WITH_APPLICATIONS_TO_ECONOMICS_SECOND_EDITION:2003} in the first step,  we have
	\[
	C_\tau=C_{\tau^\delta}\cap E=E^\delta\cap E=E,
	\]
	as desired.
	
	Hence the first five properties are equivalent.
	
	When~$\tau$ is an o-Lebesgue topology, the equivalence of~\ref{dedekind2} and~\ref{dedekind6} is a part of \cite[Theorem~4.26]{aliprantis_burkinshaw_LOCALLY_SOLID_RIESZ_SPACES_WITH_APPLICATIONS_TO_ECONOMICS_SECOND_EDITION:2003}.
\end{proof}

\begin{remark}\label{rem:fatou_topology_major_theorem}\quad
	\begin{enumerate}
		
		\item\label{rem:fatou_topology_major_theorem_1}
		The equivalence of all parts \ref{dedekind1}--\ref{dedekind6} in \cref{res:when_fatou_is_submetrisable_dedekind} does not hold for arbitrary Fatou topologies, not even when the vector lattice is Dedekind complete. Indeed, if~$A$ is an uncountable set, then, in view of \cref{res:metrisable_is_OK}, the (Fatou) supremum norm topology on $\ell^\infty(A)$ is \locsucommet,  but $\ell^\infty(A)$ does not have the countable sup property.
		
		\item\label{rem:fatou_topology_major_theorem_2}
		 In view of \cref{res:redundancy}, it follows from \cite[Theorem~5.33]{aliprantis_burkinshaw_LOCALLY_SOLID_RIESZ_SPACES_WITH_APPLICATIONS_TO_ECONOMICS_SECOND_EDITION:2003} that a vector lattice has the countable sup property when it admits a submetrisable o-Lebesgue topology. This can now be improved: as a consequence of \cref{res:when_fatou_is_submetrisable_dedekind}, \cref{res:tau_I_x_B_x}, and \cref{res:redundancy}, a vector lattice has the countable sup property when it admits an o-Lebesgue topology such that its restrictions to all principal ideals are submetrisable.		
	\end{enumerate}
	
\end{remark}

\begin{example}
	Let~$A$ be a non-empty set. We provide the Dedekind complete vector lattice $\ell^\infty(A)$ with the o-Lebesgue topology~$\tau$ that is generated by the family of lattice seminorms~$p_a$ for $a\in A$, defined by setting $p_a(x)\coloneqq\abs{x(a)}$ for $x\in\ell^\infty(A)$. Then \cref{res:when_fatou_is_submetrisable_dedekind} shows that this topology of pointwise convergence is locally (interval complete) solidly submetrisable if and only if $\ell^\infty(A)$ has the countable sup property, i.e., if and only if~$A$ is countable.
\end{example}

In view of \cref{res:when_fatou_is_submetrisable_dedekind} and, in particular, of its role as a stepping stone for \cref{res:subsequence} in the current paper, it is desirable to have sufficient conditions for a Fatou topology~$\tau$ to have full carrier~$C_\tau$. This is evidently so when~$\tau$ is metrisable, and also when the vector lattice has the countable sup property (see \cite[Theorem~4.17]{aliprantis_burkinshaw_LOCALLY_SOLID_RIESZ_SPACES_WITH_APPLICATIONS_TO_ECONOMICS_SECOND_EDITION:2003}). We shall now proceed to show that a Fatou topology~$\tau$ has full carrier when~$C_\tau$ has a countable order basis. We need the following preparatory result, for which we recall that a $\sigma$-Dedekind complete vector lattice has the principal projection property; see \cite[Theorem~25.1]{luxemburg_zaanen_RIESZ_SPACES_VOLUME_I:1971}, for example.

\begin{lemma}\label{res:countable_order_basis_and_sigma-Dedekind_complete}
	Let~$E$ be a $\sigma$-Dedekind complete vector lattice with a countable order basis $\left\{e_n:n\geq 1\right\}$ such that $0\leq e_1\leq e_2\leq \dotsc$. Denote the corresponding order projections onto the principal bands generated by the~$e_n$ by~$P_{e_n}$.  Then $x=\sup_{n\geq 1} P_{e_n}x$ for $x\in \pos{E}$.
\end{lemma}

\begin{proof}
	Take $x\in\pos{E}$. Since the~$e_n$ are increasing, so are the~$P_{e_n}$. As $0\leq P_{e_n} x\uparrow\leq x$ and~$E$ is $\sigma$-Dedekind complete, $\sup_{n\geq 1} P_{e_n} x$ exists. From the fact that $0\leq x-\sup_{n\geq 1} P_{e_n} x\leq x - P_{e_k}x$ for $k\geq 1$, we see that $(x-\sup_{n\geq 1} P_{e_n} x)\perp e_k $ for $k\geq 1$. Hence $x-\sup_{n\geq 1} P_{e_n} x=0$.
\end{proof}	

\begin{proposition}\label{res:C_tau=E_c.o.basis}Let $(E,\tau)$ be a locally solid vector lattice, where~$\tau$ is a Fatou topology. Suppose that~$C_\tau$ has a countable order basis. Then $C_\tau=E$.
\end{proposition}

\begin{proof}
	We first suppose that~$E$ is $\sigma$-Dedekind complete. According to  \cite[Theorem~4.17]{aliprantis_burkinshaw_LOCALLY_SOLID_RIESZ_SPACES_WITH_APPLICATIONS_TO_ECONOMICS_SECOND_EDITION:2003},~$C_\tau$ is an order dense $\sigma$-ideal in~$E$. Let $\left\{e_n: n\geq 1\right\}$ be a countable order basis of~$C_\tau$. By passing to $\left\{\sum_{i=1}^n \abs{e_i} : n\geq 1\right\}$, we may suppose that $0\leq e_1\leq e_2\leq\dotsc$. Since~$C_\tau$ is order dense in~$E$, $\left\{e_n: n\geq 1\right\}$ is also an order basis of~$E$.
	
	Let~$P_{e_n}$ denote the order projections from~$E$ onto the principal bands generated by the~$e_n$. Take $x\in\pos{E}$. Then \cref{res:countable_order_basis_and_sigma-Dedekind_complete} yields that $x=\sup_{n\geq 1} P_{e_n}x$ in~$E$. The formula for principal band projections and the fact that~$C_\tau$ is a $\sigma$-ideal in~$E$ imply that $P_{e_n}x\in C_\tau$. Using once more the fact that~$C_\tau$ is a $\sigma$-ideal in~$E$, we see that $x\in C_\tau$. This concludes the proof when~$E$ is $\sigma$-Dedekind complete.
	
Now we consider arbitrary~$E$, where we let~$\tau^\delta$ denote the unique extension of~$\tau$ to a Fatou topology on~$E^\delta$. Let $\left\{e_n:n\geq 1\right\}$ be a countable order basis of~$C_\tau$. Then $\left\{e_n:n\geq 1 \right\}\subseteq C_\tau\subseteq C_{\tau^\delta}$ because $C_\tau=E\cap C_{\tau^\delta}$ by \cite[Exercise~4.5]{aliprantis_burkinshaw_LOCALLY_SOLID_RIESZ_SPACES_WITH_APPLICATIONS_TO_ECONOMICS_SECOND_EDITION:2003}. Since~$E$ is order dense in~$E^\delta$, and since \cite[Theorem~4.17]{aliprantis_burkinshaw_LOCALLY_SOLID_RIESZ_SPACES_WITH_APPLICATIONS_TO_ECONOMICS_SECOND_EDITION:2003} still shows that~$C_\tau$ is order dense in~$E$,~$C_\tau$ is order dense in~$E^\delta$. Hence $\left\{e_n:n\geq 1 \right\}$ is also an order basis of~$E^\delta$, and then in particular of~$C_\tau^\delta$. The first part of the proof then shows that $C_{\tau^\delta}=E^\delta$, which implies that $C_\tau=E\cap C_{\tau^\delta}=E\cap E^\delta=E$, as required.
\end{proof}

\begin{remark}
For an o-Lebesgue topology~$\tau$, one can do better than \cref{res:C_tau=E_c.o.basis}. As seen from \cite[Lemma~6.8]{kandic_taylor:2018}, the fact that $C_\tau$ has a countable order basis then not only implies that~$\tau$ has full carrier (which is equivalent to~$E$ having the countable sup property), but also that~$E$ itself has a countable order basis.
\end{remark}
	
We combine our sufficient conditions for a Fatou topology to have full carrier with a part of  \cref{res:when_fatou_is_submetrisable_dedekind} in the following result, formulated in a way ready to be combined with \cref{res:subsequence}. 	

\begin{theorem}\label{res:when_fatou_is_locally_interval_complete_submetrisable_countability}
	Let $(E,\tau)$ be a locally solid vector lattice, where~$\tau$ is a Fatou topology. Suppose that $C_\tau=E$, which is certainly the case when~$E$ has the countable sup property, when~$C_\tau$ has a countable order basis, and when~$\tau$ is metrisable. Let $\tau^\delta$ be the extension of~$\tau$ to a Fatou topology on the Dedekind completion~$E^\delta$ of~$E$.  Then~$\tau^\delta$ is \locsucommet.
\end{theorem}

\section{Embedded unbounded order convergent sequences in topologically convergent nets}\label{sec:embedded_unbounded_order_convergent_sequences_in_convergent_nets}

\noindent
The following theorem, which will be combined with the results from \cref{sec:metrisable_and_submetrisable_topologies_on_vector_lattices} in \cref{sec:applications}, is the core result of this paper. Although the statements in its final paragraph will most likely be  sufficient in the majority of applications, we still formulate it in its most precise form.  \cref{res:subsequences_several_convergences} may give evidence that it is worth doing so; see also \cref{rem:flexibility}.

\begin{theorem}\label{res:subsequence}
	Let $(E,\tau)$ be a locally solid vector lattice, where~$\tau$ is \locsucommet, and let~$F$ be a regular vector sublattice of~$E$. Let $(x_\alpha)_{\alpha\in A}$ be a net in~$F$ and let $x\in F$ be such that $x_\alpha\tauconv x$ in~$E$.
	
	Suppose that~$A$ has a largest element $\alpha_{\mathrm{largest}}$. Then $x=x_{\alpha_{\mathrm{largest}}}$.
	
	Suppose that~$A$ has no largest element. In this case, one can choose any $\widetilde\alpha_1\in A$; then find an $\alpha_1\in A$; then choose any $\widetilde\alpha_2\in A$; then find an $\alpha_2\in A$; etc., such that:
	\begin{enumerate}
		\item\label{indices_1}
		$\alpha_1<\alpha_2<\alpha_3<\dotsc$;
		
		\item\label{indices_2}
		$\alpha_n>\widetilde\alpha_n$ for $n\geq 1$;
		
		\item\label{indices_3}
		$x_{\alpha_n}\uoconv x$ as $n\to\infty$ in~$F$.
		\end{enumerate}
	In particular, for arbitrary~$A$, there exist indices $\alpha_1\leq\alpha_2\leq\alpha_3\leq\dotsc$ such that  $x_{\alpha_n}\uoconv x$ in~$F$; if~$A$ has no largest element, then there exist indices $\alpha_1<\alpha_2<\alpha_3<\dotsc$ such that  $x_{\alpha_n}\uoconv x$ in~$F$.
\end{theorem}

\begin{proof}
	It is clear that $x=x_{\alpha_{\mathrm{largest}}}$ when~$A$ has a largest element $\alpha_{\mathrm{largest}}$.
	
	Suppose that~$A$ does not have a largest element. 
	
	Using the linearity of~$\tau$-convergence and uo-convergence; that $\abs{x_\alpha}\tauconv 0$ when $x_\alpha\tauconv 0$; that uo-convergence in~$E$ of a sequence (or net) in the regular vector sublattice~$F$ of~$E$ to an element of~$F$ implies uo-convergence in~$F$ to the same limit; and that $x_\alpha\uoconv 0$ when $\abs{x_\alpha}\uoconv 0$, it is easy to see that it is sufficient to consider the case of a net $(x_\alpha)_{\alpha\in A}$ in~$\pos{E}$ such that $x_\alpha\tauconv 0$ in~$E$.
	
	After this reduction, we choose any index $\widetilde\alpha_1\in A$. Since~$A$ has no maximal element, we can find an index~$\alpha_1$ such that $\alpha_1>\widetilde\alpha_1$. By assumption (see also \cref{rem:invariant}), we can find a translation invariant metric~$d_1$ on~$B_{x_{\alpha_1}}$ such that $d_1\left(0,y\right)\leq d_1\left(0,z\right)$ for all $y,z\in B_{x_{\alpha_1}}$ with $\abs{y}\leq\abs{z}$, and also such that:
	\begin{enumerate}[label=\textup{(\alph*)},ref=\textup{(\alph*)}]
		\item
		the metric topology on~$B_{x_{\alpha_1}}$ induced by~$d_1$ is coarser than $\tau\vert_{B_{x_{\alpha_1}}}$;
		
		\item
		$([0,x_{\alpha_1}],d_1)$ is a complete metric space.
	\end{enumerate}
	For $n=2,3,\dotsc$, we shall now inductively choose any index~$\widetilde\alpha_n$, and then find an index $\alpha_n$, together with a translation invariant metric~$d_n$ on~$B_{x_{\alpha_n}}$ such that $d_n\left(0,y\right)\leq d_n\left(0,z\right)$ for all $y,z\in B_{x_{\alpha_n}}$ with $\abs{y}\leq\abs{z}$, and also such that:
	\begin{enumerate}[label=\textup{(\roman*)},ref=\textup{(\roman*)}]
		\item\label{property_1}
	    the metric topology on~$B_{x_{\alpha_n}}$ induced by~$d_n$ is coarser than $\tau\vert_{B_{x_{\alpha_n}}}$;
	
	    \item\label{property_2}
    	$([0,x_{\alpha_n}],d_n)$ is a complete metric space;
    	
    	\item\label{property_3}
    	$\alpha_n>\alpha_{n-1}$ and $\alpha_n>\widetilde\alpha_n$;
    	
    	\item\label{property_4}
    	$d_k\left(0,x_{\alpha_n}\wedge x_{\alpha_k}\right)\leq \dfrac{1}{2^n}$ for $k=1,2,\dotsc,n-1$.
	\end{enumerate}
	We start with $n=2$. Choose any index~$\widetilde\alpha_{2}$. From the fact that $x_\alpha\tauconv 0$ in~$E$ we see that also $x_\alpha\wedge x_{\alpha_1}\tauconv 0$ because~$\tau$ is locally solid. Since $x_{\alpha}\wedge x_{\alpha_1}\in B_{x_{\alpha_1}}$ for $\alpha\in A$, and since $\tau\vert_{B_{x_{\alpha_1}}}$ is finer than the metric topology on~$B_{x_{\alpha_1}}$ induced by~$d_1$, we conclude that $d_1\left(0,x_\alpha \wedge x_{\alpha_1}\right)\to 0$. Hence we can find an $\alpha_2\in A$ with $\alpha_2>\alpha_1$ and $\alpha_2>\widetilde\alpha_2$ such that $d_1\left(0,x_{\alpha_2}\wedge x_{\alpha_1}\right)\leq \dfrac{1}{2^2}$. This takes care of~\ref{property_3} and~\ref{property_4}. Since~$\tau$ is \locsucommet, we can find a translation invariant metric~$d_2$ on $B_{x_{\alpha_2}}$ such that $d_2\left(0,y\right)\leq d_2\left(0,z\right)$ for all $y,z\in B_{x_{\alpha_2}}$ with $\abs{y}\leq\abs{z}$,  and also such that~\ref{property_1} and~\ref{property_2} are satisfied. This completes the choices for $n=2$.
	
    Suppose that $n\geq 2$ and that we have already found $\alpha_2,\dotsc,\alpha_n$
    and $d_2,\dotsc,d_n$ satisfying \ref{property_1}--\ref{property_4}. Choose any index~$\widetilde\alpha_{n+1}$.  Then~$\alpha_{n+1}$ and~$d_{n+1}$ can be found by essentially the same argument as for $n=2$. Indeed, the fact that $x_\alpha\tauconv 0$ implies that $x_{\alpha}\wedge x_{\alpha_k}\tauconv 0$ for $k=1,\dotsc,n$.  Since $x_{\alpha}\wedge x_{\alpha_k}\in B_{x_{\alpha_k}}$ for $\alpha\in A$ and $k=1,\dotsc,n$, and since $\tau\vert_{B_{x_{\alpha_k}}}$ is finer than the metric topology on $B_{x_{\alpha_k}}$ induced by~$d_k$ for $k=1,\dotsc,n$, we conclude that $d_k\left(0,x_\alpha \wedge x_{\alpha_k}\right)\to 0$ for $k=1,\dotsc,n$. Hence we can find an~$\alpha_{n+1}\in A$ with $\alpha_{n+1}>\alpha_n$ and $\alpha_{n+1}>\widetilde\alpha_{n+1}$ such that $d_k\left(0,x_{\alpha_{n+1}}\wedge x_{\alpha_k}\right)\leq \dfrac{1}{2^{n+1}}$ for $k=1,2,\dotsc,n$. This takes care of~\ref{property_3} and~\ref{property_4}. Since~$\tau$ is \locsucommet, we can find a translation invariant metric~$d_{n+1}$ on~$B_{x_{\alpha_{n+1}}}$ such that $d_{n+1}\left(0,y\right)\leq d_{n+1}\left(0,z\right)$ for all $y,z\in B_{x_{\alpha_{n+1}}}$ with $\abs{y}\leq\abs{z}$,  and also such that~\ref{property_1} and~\ref{property_2} are satisfied. This completes the induction step.

    Now that the~$\alpha_n$ and the~$d_n$ as required have been found, we fix a $k\geq 1$. From~\ref{property_4} we know that $d_k\left(0,x_{\alpha_i}\wedge x_{\alpha_k}\right)\leq \dfrac{1}{2^i}$ for $i\geq k+1$. Noticing that $\bigvee_{i=r}^s x_{\alpha_i}\wedge x_{\alpha_k}\in B_{x_{\alpha_k}}$ for $1\leq r\leq s$, using the translation invariance of~$d_k$ on~$B_{x_{\alpha_k}}$, the fact that $y\vee z-y\leq z$ for positive elements $y,z$ of a vector lattice, as well as the monotonicity of $d_k\left(0,\,\cdot\,\right)$ on~$B_{x_{\alpha_k}}^+$, we see from this that, for $p>q\geq n\geq k$,
    \begin{align*}
        d_k\left(\bigvee_{i=n}^p x_{\alpha_i}\wedge x_{\alpha_k},\bigvee_{i=n}^q x_{\alpha_i}
        \wedge x_{\alpha_k}\right)&=d_k\left(0,\bigvee_{i=n}^p x_{\alpha_i}\wedge x_{\alpha_k}-\bigvee_{i=n}^q x_{\alpha_i}\wedge x_{\alpha_k}\right)
        \\&\leq d_k\left(0,\bigvee_{i=q+1}^p x_{\alpha_i}\wedge x_{\alpha_k}\right)
        \\&\leq d_k\left(0,\sum_{i=q+1}^p x_{\alpha_i}\wedge x_{\alpha_k}\right)
        \\&\leq \sum_{i=q+1}^p d_k\left(0,x_{\alpha_i}\wedge x_{\alpha_k}\right)
        \\&\leq\sum_{i=q+1}^\infty d_k\left(0,x_{\alpha_i}\wedge x_{\alpha_k}\right)
        \\&\leq \sum_{i=q+1}^\infty \dfrac{1}{2^i}=\dfrac{1}{2^q}.
    \end{align*}
    For a fixed $n\geq k$, this implies that the sequence $\left(\bigvee_{i=n}^p x_{\alpha_i}\wedge x_{\alpha_k}\right)_{p=n}^\infty$ is a Cauchy sequence in the complete metric space $([0,x_{\alpha_k}], d_k)$. Hence its limit~$l_{n,k}$ as $p\to\infty$ exists in the topology that~$d_k$ induces on~$B_{x_{\alpha_k}}$. It follows  from the local solidness of this topology that $l_{n,k}\downarrow_{n; n\geq k}$, and also that $0\leq x_{\alpha_n}\wedge x_{\alpha_k}\leq l_{n,k}$ for $n\geq k$ ; see \cite[Theorem~2.21]{aliprantis_burkinshaw_LOCALLY_SOLID_RIESZ_SPACES_WITH_APPLICATIONS_TO_ECONOMICS_SECOND_EDITION:2003}. Furthermore, for $p\geq n\geq k+1$, we have
   \begin{align*}
   d_k\left(0,\bigvee_{i=n}^p x_{\alpha_i}\wedge x_{\alpha_k}\right)&\leq d_k\left(0,\sum_{i=n}^p  x_{\alpha_i}\wedge x_{\alpha_k}\right) \\
  &\leq \sum_{i=n}^p d_k\left(0, x_{\alpha_i}\wedge x_{\alpha_k}\right)\\
   &\leq \sum_{i=n}^p \frac{1}{2^i}\\
   &<\frac{1}{2^{n-1}}.
\end{align*}
We conclude that $d_k\left(0,l_{n,k}\right)\leq 1/2^{n-1}$ for $n\geq k+1$, so that the sequence $\left (l_{n,k}\right)_{n=k+1}^\infty$ converges to zero in the topology that~$d_k$ induces on~$B_{x_{\alpha_k}}$. As this topology is locally solid and the sequence is decreasing,  \cite[Theorem~2.21 (c)]{aliprantis_burkinshaw_LOCALLY_SOLID_RIESZ_SPACES_WITH_APPLICATIONS_TO_ECONOMICS_SECOND_EDITION:2003} shows that $l_{n,k}\downarrow_{n;n\geq k+1} 0$  in~$B_{x_{\alpha_k}}$.  Since we had already observed that $0\leq x_{\alpha_n}\wedge x_{\alpha_k}\leq l_{n,k}$ for $n\geq k$, we can now conclude that $x_{\alpha_n}\wedge x_{\alpha_k}\oconv 0$ in~$B_{x_{\alpha_k}}$ as $n\to\infty$.  Because~$B_{x_{\alpha_k}}$ is a regular vector sublattice of~$E$,  we see that  also $x_{\alpha_n}\wedge x_{\alpha_k}\oconv 0$ in~$E$ as $n\to\infty$.

After having covered one fixed $k\geq 1$, we now let~$B$ be the band in~$E$ that is generated by $\left\{x_{\alpha_k}:k=1,2,\dotsc\right\}$. Clearly, $(x_{\alpha_n})_{n=1}^\infty\subseteq B$, and we have just shown that $x_{\alpha_n}\wedge x_{\alpha_k}\oconv 0$ in~$E$ for $k=1,2,\dotsc$.  By \cite[Proposition~7.4]{deng_de_jeu:2022a}, this implies that $x_{\alpha_n}\uoconv 0$ in~$E$, as desired. This completes the proof of the existence of indices $\alpha_1,\alpha_2,\alpha_3,\dotsc$ satisfying \ref{indices_1}--\ref{indices_3} in the case that~$A$ has no largest element.

The statements in the final paragraph follow by taking all~$\alpha_n$ equal to $\alpha_{\mathrm{largest}}$ if~$A$ has a largest element~$\alpha_{\mathrm{largest}}$, and by taking all~$\widetilde\alpha_{n}$ equal to a fixed element of~$A$ if~$A$ has no largest element.
\end{proof}

The precise formulation regarding the~$\widetilde\alpha_n$ and the~$\alpha_n$ in \cref{res:subsequence} allows one to combine several convergences, as follows.

\begin{corollary}\label{res:subsequences_several_convergences}
	Let $(E,\tau)$ be a locally solid vector lattice, where~$\tau$ is \locsucommet, and let~$F$ be a regular vector sublattice of~$E$. Let $(x_\alpha)_{\alpha\in A}$ be a net in~$F$ and let $x\in F$ be such that $x_\alpha\tauconv x$ in~$E$. Suppose, furthermore, that $\tau_1,\dotsc,\tau_k$ are metrisable linear topologies on~$E$, and that $x_1,\dotsc,x_k\in E$ are such that $x_\alpha\xrightarrow{\tau_i} x_i$ for $i=1,\dotsc,k$. 
	
	Suppose that~$A$ has a largest element $\alpha_{\mathrm{largest}}$. Then $x=x_1=\dotsb=x_k=x_{\alpha_{\mathrm{largest}}}$.
	
	Suppose that~$A$ has no largest element. In this case, there exist indices $\alpha_1<\alpha_2<\alpha_3<\dotsc$ in~$A$ such that:
	\begin{enumerate}		
		\item
		$x_{\alpha_n}\uoconv x$ as $n\to\infty$ in~$F$.
		
		\item
		$x_{\alpha_n}\xrightarrow{\tau_i} x_i$ as $n\to\infty$ for $i=1,\dotsc,k$.
	\end{enumerate}
	Consequently, for arbitrary~$A$, there exist indices $\alpha_1\leq\alpha_2\leq\alpha_3\leq\dotsc$ such that  $x_{\alpha_n}\uoconv x$ in~$F$ as well as $x_{\alpha_n}\xrightarrow{\tau_i} x_i$ for $i=1,\dotsc,k$; when~$A$ has no largest element, one can choose the $\alpha_n$ to be strictly increasing. 
	\end{corollary}

\begin{proof}
	When~$A$ has a largest element, all is clear. Suppose that this not the case. 
	For $i=1,\dotsc,k$, let $\{V_n^i:n\geq 1\}$ be a $\tau_i$-neighbourhood base at~$x_i$. For each~$n$, there exists an $\widetilde\alpha_n$ such that $x_{\alpha}\in V_n^i$ for all $\alpha\geq\widetilde\alpha_n$ and all $i=1,\dots,k$. We now choose these $\widetilde\alpha_n$ in \cref{res:subsequence}.  
\end{proof}	

\begin{remark}\label{rem:flexibility}
	\cref{res:subsequences_several_convergences} does still not use the precise formulation of \cref{res:subsequence} to its full extent. The choice of the $\widetilde\alpha_n$ in \cref{res:subsequence} is flexible, in the sense that one can let the chosen value of  $\widetilde\alpha_{n+1}$ depend on the then already known values of $\alpha_1,\dotsc,\alpha_n$ and $\widetilde\alpha_1,\dotsc,\widetilde\alpha_n$. In the proof of \cref{res:subsequences_several_convergences}, however, all $\widetilde\alpha_n$ that are used in the concluding application of \cref{res:subsequence} are already given in advance. This additional flexibility is presently still awaiting an application in presumably more delicate proofs.
	\end{remark}
		
For the ease of reference in the remainder of the paper, we include the following rephrasing of the (less precise) final paragraph of \cref{res:subsequence} in the terminology of \cref{def:embedded_sequence}.

\begin{theorem}\label{res:subsequence_simplified}
	Let $(E,\tau)$ be a locally solid vector lattice, where~$\tau$ is \locsucommet, and let~$F$ be a regular vector sublattice of~$E$. Let $(x_\alpha)_{\alpha\in A}$ be a net in~$F$ and let $x\in F$ be such that $x_\alpha\tauconv x$ in~$E$. Then $(x_\alpha)_{\alpha\in A}$ contains an embedded sequence that is uo-convergent to~$x$ in~$F$.
	
	In particular, every sequence in~$F$ that is $\tau$-convergent to an element of~$F$ has a subsequence that is uo-convergent to the same limit in~$F$.
\end{theorem}

\begin{remark}\label{rem:metrisable_is_just_about_sequences}
When~$\tau$ in \cref{res:subsequence} is, in addition, metrisable, \cref{res:nets_and_sequences_general} shows that the existence of an embedded sequence in an arbitrary $\tau$-convergent net that is uo-convergent to the same limit is equivalent to the existence of a subsequence of every $\tau$-convergent sequence that is uo-convergent to the same limit. It follows from \cref{res:nets_and_sequences_no_largest_element} that they are both also equivalent to the  existence of embedded sequences in nets in the more precise formulation as in \cref{res:subsequence}.
\end{remark}

\section{Applications}\label{sec:applications}

\noindent
In this section, we combine  \cref{sec:metrisable_and_submetrisable_topologies_on_vector_lattices,sec:embedded_unbounded_order_convergent_sequences_in_convergent_nets}. While doing so, we take some care pointing out how our results relate to the existing literature.

It will become apparent that the presence of the countable sup property is a sufficient condition for several results in this section to hold. We refer to \cref{subsec:countable_sup_property} for some additional material on this property.

For the sake of simplicity, the results that are to follow are based on \cref{res:subsequence_simplified} where we take  $F=E$. We remark, however, that the more precise formulation in \cref{res:subsequence}, involving a regular vector sublattice and a more precise statement on the indices for the embedded sequence, is also always valid.\footnote{With the exception of \cref{res:subsequence_frechet}, where there appears to be no obviously valid statement for regular vector sublattices.}

We start with a case where an application of the results in \cref{sec:embedded_unbounded_order_convergent_sequences_in_convergent_nets} does, in fact, not yield an optimal statement. Let $(E,\tau)$ be a locally solid vector lattice, where~$\tau$ is metrisable and complete. The combination of \cref{res:metrisable_is_OK} and \cref{res:subsequence_simplified} shows that every $\tau$-convergent net has an embedded sequence that is uo-convergent to the same limit. A much stronger statement is true, however.

\begin{proposition}\label{res:subsequence_frechet}
	Let $(E,\tau)$ be a locally solid vector lattice, where~$\tau$ is metrisable and complete. Then every $\tau$-convergent net in~$E$ has an embedded sequence that is relatively uniformly convergent to the same limit.
	\end{proposition}

\begin{proof}
	According to \cite[Exercise~5.8]{aliprantis_burkinshaw_LOCALLY_SOLID_RIESZ_SPACES_WITH_APPLICATIONS_TO_ECONOMICS_SECOND_EDITION:2003}, every $\tau$-convergent sequence has a subsequence that is relatively uniformly convergent to the same limit. Now apply the version of \cref{res:nets_and_sequences_general} for relative uniform convergence.
	\end{proof}

For the un-topology on a Banach lattice, we have the following. The proof of  part~\ref{part:subsequence_un_2} is essentially that of \cite[Theorem~4.4]{deng_o_brien_troitsky:2017}, but formulated in such a way that it can be used in a wider context; see \cref{res:subsequence_uo-lebesgue}.

\begin{theorem}\label{res:subsequence_un}
	Let~$E$ be a Banach lattice. Then:
	\begin{enumerate}
		\item\label{part:subsequence_un_0}
			 every un-convergent net in~$E$ has an embedded sequence that is uo-conver\-gent to the same limit.
	\end{enumerate}
 If~$E$ has an order continuous norm, then:
	\begin{enumerate}[resume]
		\item\label{part:subsequence_un_1}
		every un-convergent net in~$E$ has an embedded sequence that is uo-conver\-gent as well as un-convergent to the same limit;
		
		\item\label{part:subsequence_un_2}
		a sequence in~$E$ is un-convergent to $x\in E$ if and only if every subsequence has a further subsequence that is uo-convergent to~$x$.
	\end{enumerate}
\end{theorem}

\begin{proof}
	The statement in~\ref{part:subsequence_un_0} is immediate from \cref{res:un_is_l.c.sub} and \cref{res:subsequence_simplified}.
	
	Suppose that~$E$ has an order continuous norm.
	
	In this case, the un-topology is the uo-Lebesgue topology on~$E$ (see \cite[Proposition~2.5]{deng_o_brien_troitsky:2017}, for example), which makes clear that~\ref{part:subsequence_un_1} follows from~\ref{part:subsequence_un_0}.
	
	The forward implication in~\ref{part:subsequence_un_2} is immediate from~\ref{part:subsequence_un_0}. For the converse implication, suppose that a sequence $(x_n)_{n\geq 1}$ has the property that every subsequence has a further subsequence that is uo-convergent to~$x$, but that  $(x_n)_{n\geq 1}$ is not un-convergent to~$x$. In this case, there exists an un-neighbourhood~$V$ of~$x$ and a subsequence of $(x_n)_{n\geq 1}$ that stays outside~$V$. By assumption, there is a further subsequence that is uo-convergent to~$x$. Since the un-topology is the uo-Lebesgue topology on~$E$, this further subsequence is eventually in~$V$. This is a contradiction.
	\end{proof}

\begin{remark}
	Part~\ref{part:subsequence_un_0} of \cref{res:subsequence_un} improves \cite[Proposition~4.1]{deng_o_brien_troitsky:2017}, where it is shown that every un-convergent \emph{sequence} in a Banach lattice has a subsequence that is uo-convergent to the same limit. The statement in~\ref{part:subsequence_un_1} is \cite[Corollary~3.5]{deng_o_brien_troitsky:2017}. The statement in~\ref{part:subsequence_un_2} is \cite[Theorem~4.4]{deng_o_brien_troitsky:2017}.
\end{remark}

We now turn to Fatou topologies with full carrier, for which we have the following basic result.

\begin{theorem}\label{res:subsequence_fatou}
	Let $(E,\tau)$ be a locally solid vector lattice, where~$\tau$ is a Fatou topology. Suppose that $C_\tau=E$, which is certainly the case when~$E$ has the countable sup property, when~$C_\tau$ has a countable order basis, and when~$\tau$ is metrisable. Then every $\tau$-convergent net in~$E$ has an embedded sequence that is uo-convergent to the same limit.
\end{theorem}

\begin{proof}
	Suppose that $(x_\alpha)_{\alpha\in A}$ is a net in~$E$ and that $x_\alpha\tauconv x$ for some $x\in E$. If we let~$\tau^\delta$ denote the extension of~$\tau$ to a Fatou topology on~$E^\delta$, then also $x_\alpha\conv{\tau^\delta}x$ in~$E^\delta$. The combination of \cref{res:when_fatou_is_locally_interval_complete_submetrisable_countability} and \cref{res:subsequence_simplified} yields that $(x_\alpha)_{\alpha\in A}$ has an embedded sequence that is uo-convergent to~$x$ in~$E^\delta$. As~$E$ is a regular vector sublattice of~$E^\delta$, this embedded sequence is also uo-convergent to~$x$ in~$E$.
	\end{proof}

\begin{remark}\quad
	\begin{enumerate}
		\item 	
		The special case of \cref{res:subsequence_fatou} where $C_\tau$ has a countable order basis is \cite[Theorem~6.7]{kandic_taylor:2018}.
	
		\item
		According to \cite[Theorem~4.19]{aliprantis_burkinshaw_LOCALLY_SOLID_RIESZ_SPACES_WITH_APPLICATIONS_TO_ECONOMICS_SECOND_EDITION:2003}, an order bounded convergent net in a Fatou topology on a vector lattice with the countable sup property has an embedded sequence that is order convergent to the same limit. This is clear from \cref{res:subsequence_fatou}, since an order bounded uo-convergent sequence is order convergent. It is, in fact, more generally true when the topology has full carrier.
	\end{enumerate}
\end{remark}

Before continuing with our applications, we note that, for o-Lebesgue topologies, there is a converse to \cref{res:subsequence_fatou}. In fact, we have the following.

\begin{theorem}\label{res:equivalence_for_o-lebesgue_topologies}
	Let $(E,\tau)$ be a locally solid vector lattice, where~$\tau$ is an o-Lebesgue topology. The following are equivalent:
	\begin{enumerate}
		
		\item\label{part:equivalence_for_o-lebesgue_topologies_1}
		$C_\tau=E$;
		
		\item\label{part:equivalence_for_o-lebesgue_topologies_2}
		$E$ has the countable sup property;		
				
		\item\label{part:equivalence_for_o-lebesgue_topologies_3}
		every $\tau$-convergent net in~$E$ has an embedded sequence that is uo-conver\-gent to the same limit;
						
		\item\label{part:equivalence_for_o-lebesgue_topologies_4}
		every increasing $\tau$-convergent net in~$\pos{E}$ has an embedded sequence that is uo-convergent to the same limit. 
		\end{enumerate}		
\end{theorem}

\begin{proof}
The equivalence of~\ref{part:equivalence_for_o-lebesgue_topologies_1} and ~\ref{part:equivalence_for_o-lebesgue_topologies_2}, taken from the literature, was already in \cref{res:when_fatou_is_submetrisable_dedekind}. \cref{res:subsequence_fatou} shows that ~\ref{part:equivalence_for_o-lebesgue_topologies_1} implies~\ref{part:equivalence_for_o-lebesgue_topologies_3}, and it is clear that~\ref{part:equivalence_for_o-lebesgue_topologies_3} implies~\ref{part:equivalence_for_o-lebesgue_topologies_4}.

We show that~\ref{part:equivalence_for_o-lebesgue_topologies_4} implies~\ref{part:equivalence_for_o-lebesgue_topologies_2}. Suppose that $(x_\alpha)_{\alpha\in A}$ is a net in~$\pos{E}$ and that $x_\alpha\uparrow x$ for some $x\in\pos{E}$. Then $x_\alpha\tauconv x$. By assumption, there exists an embedded sequence $(x_{\alpha_n})_{n\geq 1}$ such that $x_{\alpha_n}\uoconv x$ as $n\to\infty$. Since the sequence is order bounded, we even have $x_{\alpha_n}\oconv x$. As the $\alpha_n$ are increasing, so are the $x_{\alpha_n}$, and we conclude that $x_{\alpha_n}\uparrow x$. Now \cite[Theorem~23.2.(iii)]{luxemburg_zaanen_RIESZ_SPACES_VOLUME_I:1971} shows that~$E$ has the countable sup property.
\end{proof}

\begin{remark}\quad
	\begin{enumerate}
	
	\item 
	We see that, for o-Lebesgue topologies, the existence of embedded uo-convergent subsequences in the restricted class of topologically convergent nets in~\ref{part:equivalence_for_o-lebesgue_topologies_4} already implies this existence in general topologically convergent nets in the more precise form in \cref{res:subsequence}.
	\item 
	For arbitrary Fatou topologies, \cref{res:equivalence_for_o-lebesgue_topologies} does not hold. Indeed, the combination of part~\ref{rem:fatou_topology_major_theorem_1} of \cref{rem:fatou_topology_major_theorem} and \cref{res:subsequence_fatou} shows that ~\ref{part:equivalence_for_o-lebesgue_topologies_3} does not imply ~\ref{part:equivalence_for_o-lebesgue_topologies_2} within this larger class of topologies. For Fatou topologies, it would be satisfactory if, for example,~\ref{part:equivalence_for_o-lebesgue_topologies_1} could still be shown to be equivalent to~\ref{part:equivalence_for_o-lebesgue_topologies_3}, or perhaps to the more precise version of the existence of embedded uo-convergent sequences as in \cref{res:subsequence}. In support of this, we can, however, presently only mention that the (metrisable) Fatou topology in part~\ref{rem:fatou_topology_major_theorem_1} of \cref{rem:fatou_topology_major_theorem} indeed has full carrier.
\end{enumerate}
\end{remark}	

Continuing with our applications, we now specialise within the Fatou topologies to the uo-Lebesgue topologies. Employing arguments as in the proof of the parts~\ref{part:subsequence_un_1} and~\ref{part:subsequence_un_2} of \cref{res:subsequence_un}, and recalling that, for o-Lebesgue topologies, having full carrier is equivalent to the vector lattice having the countable sup property, one obtains the following result.

\begin{theorem}\label{res:subsequence_uo-lebesgue}
	Let $(E,\tau)$ be a locally solid vector lattice, where~$\tau$ is a uo-Lebesgue topology. Suppose that $C_\tau=E$ or, equivalently, that~$E$ has the countable sup property; this is certainly the case when~$C_\tau$ has a countable order basis, and when~$\tau$ is metrisable. Then:
		\begin{enumerate}
			\item\label{part:subsequence_uo-lebesgue_1}
			every $\tau$-convergent net in~$E$ has an embedded sequence that is uo-conver\-gent as well as $\tau$-convergent to the same limit;
			
			\item\label{part:subsequence_uo-lebesgue_2}
			a sequence in~$E$ is $\tau$-convergent to $x\in E$ if and only if every subsequence has a further subsequence that is uo-convergent to~$x$.				
		\end{enumerate}
\end{theorem}

\begin{remark}\label{rem:subsequence_uo-lebesgue}\quad
	\begin{enumerate}
		\item
		Suppose that~$E$ is a Banach lattice, and let~$\tau$ denote its norm topology. Then $C_{\unb\tau}=E$ by \cref{res:carrier_of_induced_unbounded_topology}. If~$E$ has an order continuous norm (which implies that it has the countable sup property), then $\unb\tau$ is the uo-Lebesgue topology on~$E$. We thus see that \cref{res:subsequence_uo-lebesgue} yields the parts~\ref{part:subsequence_un_1} and~\ref{part:subsequence_un_2} of \cref{res:subsequence_un} again, albeit not via the shortest route.		
		\item
		It follows from \cite[Lemma~6.8 and Corollary~6.9]{kandic_taylor:2018} that part~\ref{part:subsequence_uo-lebesgue_2} of \cref{res:subsequence_uo-lebesgue} holds when~$E$ has the countable sup property and a countable order basis. \cref{res:subsequence_uo-lebesgue}, however, shows that~$E$ having the countable sup property is alone already sufficient for part~\ref{part:subsequence_uo-lebesgue_2} to hold.
		
		\item
		For the sake of completeness we recall from \cite[Theorem~4.9]{kandic_taylor:2018} that a uo-Lebesgue topology on a vector lattice is metrisable if and only if the vector lattice has the countable sup property and a countable order basis.
	\end{enumerate}
\end{remark}

\cref{res:subsequence_uo-lebesgue} has the following consequence.
\begin{corollary}\label{res:subsequence_uo-lebesgue_induced}
	Let~$E$ be a vector lattice, and suppose that~$F$ is an order dense ideal in~$E$ admitting an o-Lebesgue topology~$\tau_F$. Then~$\unb_F\tau_F$ is a uo-Lebesgue topology on~$E$.
	Suppose that at least one of the following is satisfied:
	\begin{enumerate}[label=\textup{(\alph*)},ref=\textup{(\alph*)}]
		\item\label{part:subsequence_uo-lebesgue_induced_hypothesis_1}
		$E$ has the countable sup property;
	
		\item\label{part:subsequence_uo-lebesgue_induced_hypothesis_2}
		$\tau_F$ is metrisable and~$F$ has a countable order basis \uppars{which is always the case when~$\tau_F$ is a metrisable uo-Lebesgue topology};
	\end{enumerate}
	Then $C_{\unb_F\tau_F}=E$. Consequently:
	\begin{enumerate}
		\item\label{part:subsequence_uo-lebesgue_induced_1}
		every $\unb_F\tau_F$-convergent net in~$E$ has an embedded sequence that is uo-convergent as well as $\unb_F\tau_F$-convergent to the same limit;
		
		\item\label{part:subsequence_uo-lebesgue_induced_2}
		a sequence in~$E$ is $\unb_F\tau_F$-convergent to $x\in E$ if and only if every subsequence has a further subsequence that is uo-convergent to~$x$.
	\end{enumerate}
\end{corollary}

\begin{proof}
It follows from \cite[Proposition~4.1]{deng_de_jeu:2022a} that $\unb_F\tau_F$ is a uo-Lebesgue topology on~$E$. By \cite[Theorem~4.9]{kandic_taylor:2018},~$F$ has a countable order basis when~$\tau_F$ is a metrisable uo-Lebesgue topology.

Suppose that~$\tau_F$ is metrisable and that~$F$ has a countable order basis. According to \cref{res:carrier_of_induced_unbounded_topology}, we have $F\subseteq C_{\unb_F\tau_F}$. Since~$F$ is then an order dense vector sublattice of $C_{\unb_F\tau_F}$, $C_{\unb_F\tau_F}$ also has a countable order basis. Now we can apply \cref{res:subsequence_uo-lebesgue}.
\end{proof}

\begin{remark}\quad
\begin{enumerate}
	\item
	When~$E$ is a vector lattice containing a Banach lattice~$F$ with an order continuous norm as an order dense ideal, \cref{res:subsequence_uo-lebesgue_induced} specifies to a considerable improvement of \cite[Theorem~9.5]{kandic_li_troitsky:2018}. An inspection of \cite[Example~9.6]{kandic_li_troitsky:2018} shows that the condition in~\ref{part:subsequence_uo-lebesgue_induced_hypothesis_2} that~$F$ have a countable order basis cannot be omitted.
	
	\item
	A vector lattice~$E$ with the property that~$\ocdual{E}$ separates its points admits an o-Lebesgue topology; see \cite[Lemma~5.1]{deng_de_jeu:2022a}, for example. With this in mind, it is easy to see that \cref{res:subsequence_uo-lebesgue_induced} implies \cite[Theorems~7.8 and~7.13]{deng_de_jeu:2022a}.
	\end{enumerate}
\end{remark}

We continue with a consequence of \cref{res:subsequence_uo-lebesgue} in measure theory.

\begin{corollary}\label{res:subsequence_measure_spaces}
	Let $(X,\Sigma,\mu)$ be a semi-finite measure space, and let~$E$ be the ideal in $\Ell^0(X,\Sigma,\mu)$ consisting of \uppars{equivalence classes of} measurable functions with $\sigma$-finite supports. Then~$E$ admits a uo-Lebesgue topology, which is the topology of local convergence in measure, and~$E$ has the countable sup property. Consequently:
	\begin{enumerate}
		\item\label{part:subsequence_measure_spaces_1}
		every $\mu^\ast$-convergent net in~$E$ has an embedded sequence that is almost everywhere convergent as well as $\mu^\ast$-convergent to the same limit;
		
		\item\label{part:subsequence_measure_spaces_2}
		a sequence in~$E$ is $\mu^\ast$-convergent to $f\in E$ if and only if every subsequence has a further subsequence that is almost everywhere convergent to~$f$.
	\end{enumerate}
\end{corollary}

\begin{proof}
	It was already observed in \cite[p.292]{conradie:2005} that $\Ell^0(X,\Sigma,\mu)$ has a uo-Lebesgue topology, and that it is the topology of local convergence in measure. A more detailed statement and its proof can be found as \cite[Theorem~6.1]{deng_de_jeu:2022a}. The restriction of this topology to the regular vector sublattice~$E$ is a uo-Lebesgue topology on~$E$. A moment's thought shows that the $\sigma$-finiteness of the supports of the elements of~$E$ implies that every disjoint system of positive elements of~$E$ which is bounded from above in~$E$ is at most countable. Hence~$E$ has the countable sup property by \cite[Theorem~29.3(vi)]{luxemburg_zaanen_RIESZ_SPACES_VOLUME_I:1971}. We can now apply   \cref{res:subsequence_uo-lebesgue}. The proof is then completed by using that, for arbitrary measure spaces, uo-convergence of sequences in regular vector sublattices of $\Ell^0(X,\Sigma,\mu)$ is equivalent to convergence almost everywhere; see \cite[Remark~3.4]{gao_troitsky_xanthos:2017}.
\end{proof}

\begin{remark}\quad
	\begin{enumerate}
		\item
		When~$\mu$ is $\sigma$-finite, one has~$E=\Ell^0(X,\Sigma,\mu)$ in \cref{res:subsequence_measure_spaces}. In this case, part~\ref{part:subsequence_measure_spaces_1} yields \cite[Theorem~7.11]{deng_de_jeu:2022a}, and  part~\ref{part:subsequence_measure_spaces_2} yields \cite[245K]{fremlin_MEASURE_THEORY_VOLUME_2:2003}.
		
		\item For the sake of completeness we recall (as was already mentioned in \cref{sec:introduction_and_overview}) that, for arbitrary measure spaces, a sequence in~$\Ell^0(X,\Sigma,\mu)$ that is (globally) convergent in measure has a subsequence that converges almost everywhere to the same limit; see \cite[Theorem~19.3]{aliprantis_burkinshaw_PRINCIPLES_OF_REAL_ANALYSIS_THIRD_EDITION:1998} or \cite[Theorem~2.30]{folland_REAL_ANALYSIS_SECOND_EDITION:1999}, for example. This goes back to F.\ Riesz in \cite{riesz:1909}, correcting an earlier statement by Lebesgue in \cite{lebesgue_LECONS_SUR_LES_SERIES_TRIGONOMETRIQUES:1906}.	
	\end{enumerate}
\end{remark}

We continue with an application of \cref{res:subsequence_uo-lebesgue} to weak closures. For this, we need some terminology. Suppose that~$E$ is a vector lattice, and that~$I$ is an ideal in~$\ocdual{E}$ that separates the points of~$E$. In this case, the lattice seminorms~$p_\varphi$, defined for $\varphi\in I$ by setting $p_\varphi(x)\coloneqq\abs{\varphi}(\abs{x})$ for $x\in E$, define a locally convex o-Lebesgue topology $\abs{\sigma}(E,I)$ on~$E$. The uo-Lebesgue topology  $\unb\abs{\sigma}(E,I)$ on~$E$, resulting from this so-called absolute weak topology $\abs{\sigma}(E,I)$ on~$E$ that is induced by~$I$, plays a part in the following result.

\begin{corollary}\label{res:convex}
	Let~$E$ be a vector lattice, and suppose that~$\ocdual{E}$ has an ideal~$I$ that separates the points of~$E$.  Suppose that $C_{\unb\abs{\sigma}(E,I)}=E$ or, equivalently, that~$E$ has the countable sup property; this is certainly the case when $C_{\unb\abs{\sigma}(E,I)}$ has a countable order basis, and when $\unb\abs{\sigma}(E,I)$ is metrisable. Let~$S$ be a convex subset of~$E$, and take $x\in \overline{S}^{\sigma(E,I)}$. Then there exists a sequence in~$S$ that is uo-convergent to~$x$.
\end{corollary}

\begin{proof}
	By Kaplan's theorem (see \cite[Theorem~2.33]{aliprantis_burkinshaw_LOCALLY_SOLID_RIESZ_SPACES_WITH_APPLICATIONS_TO_ECONOMICS_SECOND_EDITION:2003}, for example), the topological duals of the locally convex topological vector spaces $(E, \abs{\sigma}(E,I))$ and $(E, \sigma(E,I))$ are both equal to~$I$. Hence $\overline{S}^{\sigma(E,I)}=\overline{S}^{\abs{\sigma}(E,I)}$, so that there exists a net $(x_\alpha)_{\alpha\in A}$ in~$S$ with $x_\alpha\conv{\abs{\sigma}(E,I)} x$. Then certainly $x_\alpha\conv{\unb\abs{\sigma}(E,I)} x$. We can now apply \cref{res:subsequence_uo-lebesgue}.
\end{proof}

\begin{remark}
	In \cite[Theorem~4.1]{gao_leung_xanthos:2018}, the existence of a uo-convergent sequence as in \cref{res:convex} was established under the single hypothesis that the ideal $I$ in $\ocdual{E}$ contain  a strictly positive linear functional. This is a special case of \cref{res:convex}. Indeed, if a vector lattice has a strictly positive linear functional~$\varphi$, then it has the countable sup property by \cite[Theorem~1.45]{aliprantis_burkinshaw_LOCALLY_SOLID_RIESZ_SPACES_WITH_APPLICATIONS_TO_ECONOMICS_SECOND_EDITION:2003}.  Furthermore, the principal ideal in~$\odual{E}$ generated by~$\varphi$ then already separates the points of~$E$ by \cite[Proposition~2.1(2)]{deng_de_jeu:2022a}, so this is then certainly true for~$I$.
	
	We refer to \cite{gao_leung_xanthos:2018} for further elaborations on this special case.
\end{remark}

We continue with an application of \cref{res:subsequence_uo-lebesgue} to adherences and closures. As a preparation, we recall some notation from \cite[Section~8]{deng_de_jeu:2022b}. Let~$A$ be a subset of a vector lattice~$E$. We write
\[
\suoadh{A}=\left\{x\in E: \text{ there exists a sequence }(x_n)_{n\geq 1}\text{ in } A \text{ with }x_n\uoconv x \text{ in } E\right\}
\]
and
\[
\uoadh{A}=\left\{x\in E: \text{ there exists a net }(x_\alpha)_\alpha\text{ in } A \text{ with }x_\alpha\uoconv x  \text{ in } E\right\}
\]
for the \emph{\suoadhtext} and \emph{\uoadhtext} of~$A$ in~$E$, respectively. When~$\tau$ is a topology on~$E$, one similarly defines the $\sigma$-$\tau$-adherence and the $\tau$-adherence of~$A$, the latter simply being its $\tau$-closure~$\overline{A}^\tau$.

A subset~$A$ of~$E$ is said to be \emph{$\sigma$-uo-closed} when $\suoadh{A}=A$. The $\sigma$-uo-closed subsets of~$E$ are the closed subsets of a topology on~$E$, which is called the $\sigma$-uo-topology. The closure of a subset~$A$ in this topology is denoted by $\suoclos{A}$. We have $\suoclos{\suoadh{A}}=\suoclos{A}$, and~$A$ is closed in the $\sigma$-uo-topology if and only if $\suoadh{A}=\suoclos{A}$. We refer to \cite[Lemmas~2.3 and~2.4]{deng_de_jeu:2022b} for details. Similar statements hold for uo-adherences and $\sigma$-$\tau$-adherences and the associated topologies.

The conclusion and the proof of the following result are completely analogous to those of \cite[Theorem~8.1]{deng_de_jeu:2022b}. Its hypotheses are considerably weaker, however, since we now have \cref{res:subsequence_uo-lebesgue} available, rather than \cite[Theorem~7.8]{deng_de_jeu:2022b} with its much more restrictive hypotheses.

\begin{theorem}\label{res:seven_sets_equal}
Let~$E$ be a vector lattice that admits a uo-Lebesgue topology~$\tau$, and let~$A$ be a subset of~$E$. Suppose that $C_\tau=E$ or, equivalently, that~$E$ has the countable sup property; this is certainly the case when~$C_\tau$ has a countable order basis, and when~$\tau$ is metrisable.  Then the following seven subsets of~$E$ are all equal:
\begin{enumerate}
	\item
	$\sadh{\tau}{A}$ and $\overline{A}^{\sigma\text{-}\tau}$;
	
	\item
	$\suoadh{A}$ and $\suoclos{A}$;
	
	\item
	$\uoadh{A}$ and $\uoclos{A}$;
	
	\item
	$\overline{A}^\tau$.
\end{enumerate}

In particular, the $\sigma$-$\tau$-topology, the $\sigma$-uo-topology, and the uo-topology on~$E$ all coincide with~$\tau$.
\end{theorem}

On combining \cref{res:subsequence_measure_spaces} and \cref{res:seven_sets_equal}, and recalling again from \cite[Remark~3.4]{gao_troitsky_xanthos:2017} that uo-convergence of sequences in regular vector   sublattices of $\Ell^0(X,\Sigma,\mu)$ is the same as almost everywhere convergence, we obtain the following.

 \begin{corollary}\label{res:subset_fremlind_improved}
 	Let $(X,\Sigma,\mu)$ be a semi-finite measure space, and let~$E$ be the ideal in $\Ell^0(X,\Sigma,\mu)$ consisting of \uppars{equivalence classes of} measurable functions with $\sigma$-finite supports. Let~$A$ be a subset of~$E$. Then~$A$ is a closed subset of~$E$ in the topology of local convergence in measure if and only if it contains the almost everywhere limits of sequences in~$A$.
 \end{corollary}

\begin{remark}
	For a $\sigma$-finite measure, \cref{res:subset_fremlind_improved} is \cite[245L(b)]{fremlin_MEASURE_THEORY_VOLUME_2:2003}.
\end{remark}

\subsection*{Acknowledgements} The authors thank Evgeny Bilokopytov for pointing out that an earlier version of \cref{res:when_is_submetrisable} could be strengthened to the current one.

\bibliographystyle{plain}
\bibliography{general_bibliography}

\end{document}